\documentclass[10pt]{article}
\usepackage{amsmath}
\usepackage{amsfonts}
\usepackage{amssymb}
\usepackage{graphicx}
\usepackage{mathrsfs}
\usepackage{amsthm}
\usepackage{xcolor}
\usepackage{bbm}
\usepackage{verbatim}
\usepackage[T1]{fontenc}
\usepackage{dsfont}
\usepackage{mathrsfs}
\usepackage[body={15.5cm,21cm}, top=3cm]{geometry}
\usepackage{paralist}
\usepackage{indentfirst}
\usepackage{hyperref}
\hypersetup{hypertex = true,
colorlinks=true,
linkcolor=blue,
anchorcolor=blue,
citecolor =blue }
\providecommand{\U}[1]{\protect \rule{.1in}{.1in}}
\numberwithin{equation}{section}
\newtheorem{Theorem}{Theorem}[section]

\newtheorem{definition}[Theorem]{Definition}

\newtheorem{lemma}[Theorem]{Lemma}

\newtheorem{proposition}[Theorem]{Proposition}
\newtheorem{remark}[Theorem]{Remark}

\newtheorem{assumption}[Theorem]{Assumption}
\newtheorem*{assumption*}{Assumption}
\allowdisplaybreaks
\begin{document}
\title{Infinite horizon McKean-Vlasov FBSDEs and applications to mean field  control problems}
\author{Tianjiao Hua \thanks{School of Mathematical Sciences, Shanghai Jiao Tong University, China (htj960127@sjtu.edu.cn)}
\and
Peng Luo \thanks{School of Mathematical Sciences, Shanghai Jiao Tong University, China (peng.luo@sjtu.edu.cn)}}
\maketitle
\begin{abstract}
    In this paper, we study a class of infinite horizon fully coupled McKean-Vlasov forward-backward stochastic differential equations (FBSDEs).  We propose a generalized monotonicity condition involving two flexible functions. Under this condition, we establish the well-posedness results for infinite horizon McKean-Vlasov FBSDEs by the method of continuation, including the unique solvability, an estimate of
the solution, and the related continuous dependence property of the solution on the coefficients. Based on the solvability result, we study an infinite horizon mean field control problem. Moreover, by choosing appropriate form of the flexible functions, we can eliminate the different phenomenon between the linear-quadratic (LQ) problems on infinite horizon and finite horizon proposed in Wei and Yu (SIAM J. Control Optim. 59: 2594–2623, 2021). 
\end{abstract}
\textbf{MSC-classification}: 60H10, 93H20, 49N80.
\section{Introduction}  
Mean field control problems have, in the recent years, drawn the attention of the applied mathematics community. Being an extension of the classical optimal control, it has been studied from different angles. The first one is to use the so-called dynamic programming principle (DPP). Compared with the classical control problems, the presence of the distribution of the controlled process in the coefficients brings additional difficulty. One can refer to \cite{2018DPP,Pham2017dpp,pham2018dpp} for related research on establishing DPP for mean field control problems. A second way is based on the Pontryagin's maximum principle. This approach has been successfully developed in many literature, see \cite{carmonamkfbsde,lijuan,Andersson2011,meanfield2018book1}. Recently, Bayraktar and Zhang \cite{bayraktar2023solvability} studied a  mean field control problem on infinite horizon using FBSDE techniques, where the state volatility is a constant.

In this paper, we consider the following infinite horizon mean field control problem: Minimize
\begin{equation}\label{cost functional}
J(\alpha):=\mathbb{E}\left[\int_0^{\infty} e^{2 K t} f\left(t, X_t, \mathcal{L}\left(X_t\right), \alpha_t\right) d t\right]    
\end{equation}
subject to
\begin{equation}\label{state dynamics}
  \left\{  \begin{aligned}
dX_t&=b\left(t, X_t, \mathcal{L}\left(X_t\right), \alpha_t\right) dt+\sigma(t,X_{t},\mathcal{L}\left(X_t\right),\alpha_{t}) d W_t, \quad t\in[0,\infty),\\
X_0&=\xi, 
\end{aligned}\right.
\end{equation}
where $K\in \mathbb{R}$ is a constant, $W_{t}$ is a $d$-dimensional Brownian motion, $(b,f,\sigma):[0,\infty)\times \mathbb{R}^{n} \times \mathcal{P}_{2}(\mathbb{R}^{n}) \times A \rightarrow \mathbb{R}^{n}\times \mathbb{R} \times \mathbb{R}^{n\times d}$ are measurable mappings, and $\alpha=\left(\alpha_t\right)_{t
\geq 0 }$ is a progressively measurable process with values in a measurable space $(A, \mathcal{A})$. It can be noticed that we introduce a term $e^{2Kt}$ to "suppress" the growth of the running cost function $f(t,X(t),\mathcal{L}(X_{t}),\alpha_{t})$, thereby obtaining the well-posedness of the cost functional \eqref{cost functional}. In the literature \cite{yu2021,wei2023infinite,bayraktar2023solvability}, they used the same technique when studying infinite horizon control problems. As indicated in \cite{yu2021}, the choice of parameter $K$ depends on the properties of the coefficients. Specifically, if the monotonicity of $b$ with respect to $x$ is sufficiently negative, the parameter $K$ may take a positive real number. Therefore, the appearance of parameter $K$ is friendly with various models and we can always identify the parameter $K$ according to the the intrinsic properties of coefficients and formulate the control problems.

Following the probabilistic approach to finite horizon mean field  control problems, we establish an appropriate form of the Pontryagin's maximum principle for the infinite horizon case and then the infinite horizon mean field control problem \eqref{cost functional}-\eqref{state dynamics} is reduced to solving an infinite horizon McKean-Vlasov FBSDE, which is also called a Hamiltonian system. Motivated by this, we aim to establish the well-posedness of a more general form of infinite horizon coupled McKean-Vlasov FBSDEs. In detail, we consider the following infinite horizon McKean-Vlasov FBSDE:
\begin{equation}\label{FBSDE}
\left\{\begin{aligned}
d X_t&=B\left(t, X_t, Y_t, Z_{t},\mathcal{L}\left(X_t, Y_t, Z_{t}\right)\right)  dt+\sigma\left(t, X_t, Y_t, Z_{t},\mathcal{L}\left(X_t, Y_t, Z_{t}\right)\right)  d W_t, \quad t\in[0,\infty),\\
d Y_t&=F\left(t, X_t, Y_t, Z_{t}, \mathcal{L}\left(X_t, Y_t, Z_{t}\right)\right)  dt+Z_t  d W_t, \quad t \in [0,\infty), \\
X_0&=\xi,
\end{aligned}\right.    
\end{equation}
where $W_{t}$ is a d-dimensional Brownian motion,  $X, Y, Z$ take values in $\mathbb{R}^n, \mathbb{R}^m, \mathbb{R}^{m \times d}$, $B, \sigma,F$ are progressively measurable functions with appropriate dimensions, $\xi$ is an $\mathcal{F}_{0}$-measurable square integrable random variable and  $\mathcal{L}\left(X_t, Y_t, Z_t\right)$ denote the probability measures induced by $\left(X_t, Y_t, Z_t\right)$. We aim to look for solutions $(X,Y,Z)$ to \eqref{FBSDE} in $L^{2,K}_{\mathbb{F}}(0,\infty;\mathbb{R}^{n+m+m\times d})$, where $K\in \mathbb{R}$ and $L^{2,K}_{\mathbb{F}}(0,\infty;\mathbb{R}^{n})$ is the Hilbert space of all $\mathbb{R}^{n}$-valued adapted processes $\nu_{t}$ such that
 \begin{equation}\label{solution space}
     \mathbb{E}\int_{0}^{\infty} |e^{Kt}\nu_{t}|^{2} dt < \infty.
\end{equation}
 
Finite horizon classical FBSDEs were first investigated by Antonelli \cite{antonelli1993backward} and a local existence and uniqueness result was obtained. For the global solvability results, there exist four main methods: the method of contraction mapping used by Pardoux and Tang \cite{tang1999}, four-step scheme approach introduced by Ma, Protter and Yong \cite{ma1994}, continuation method initiated by Hu and Peng \cite{hu1995}, Peng and Wu \cite{peng1999fully} and improved by Yong \cite{yong1997,yong2010}, random decoupling field introduced by Ma et al. \cite{ma2015} and extended by Fromm and Imkeller \cite{fromm2013existence} and Hua and Luo \cite{hua2022unified}. For more detailed results on finite horizon FBSDEs, one can refer to the monograph by Ma and Yong \cite{mabook2007}. The research on finite horizon mean field FBSDEs builds upon these methods and  further developments have been made, see \cite{bensoussan2015well,carmonamkfbsde,meanfield2018book1,carmonamffbsde,hua2023well}.

In \cite{peng2000}, Peng and Shi, for the first time, investigated fully coupled classical infinite horizon FBSDEs by the method of continuation. Later, Wu \cite{wu2003fully} studied this problem in some different monotonicity framework. Yin \cite{yin2008solutions,yin2011forward} studied the same issue by the method of contraction mapping. Shi and Zhao \cite{shi2020} extended the results \cite{peng2000} to a larger space and studied the connection of infinite horizon FBSDEs with corresponding PDEs. Besides, Yu \cite{yu2017} investigated infinite horizon FBSDEs driven by both Brownian motions and Poisson processes. Bayraktar and Zhang \cite{bayraktar2023solvability} extended the infinite horizon FBSDEs results to a type of infinite horizon McKean-Vlasov FBSDEs where the coefficients $B,F$ depend on $(X,Y,\mathcal{L}(X,Y))$ and the coefficient $\sigma$ is a constant. Recently, Wei, Xu and Yu \cite{wei2023infinite} studied a kind of infinite horizon linear mean field type FBSDEs with jumps.

In this paper, we establish an existence and uniqueness result and a pair of estimates for the solutions to infinite horizon McKean-Vlasov FBSDE \eqref{FBSDE} with the method of continuation.
 The key point of the method of continuation lies in proposing suitable motononicity condition. For the finite horizon McKean-Vlasov FBSDEs, Bensoussan et al. \cite{bensoussan2015well} extended the results in Peng and Wu \cite{peng1999fully} to include mean-field terms by proposing the following motononicity condition:
$$
 \begin{aligned}
& \mathbb{E}\left[\left \langle \Gamma\left(t ; \Theta_{1}, \mathcal{L}(\Theta_{1})\right)-\Gamma\left(t ; \Theta_2, \mathcal{L}(\Theta_{2})\right), \Theta_1-\Theta_2\right\rangle \right] \\
\leq & -\beta_1 \mathbb{E}\left[\left|G (X_{1}-X_{2})\right|^2\right]-\beta_2\left(\mathbb{E}\left[\left|G^{\top} (Y_{1}-Y_{2})\right|^2\right]+\mathbb{E}\left[\left|G^{\top}(Z_{1}-Z_{2})\right|^2\right]\right),
\end{aligned}
$$
for any $\Theta_{1} = (X_{1},Y_{1},Z_{1}), \Theta_{2} = (X_{2},Y_{2},Z_{2})\in L^{2}(\mathbb{R}^{n+m+m \times d})$, where $\Gamma=(G^{\top}F,GB,G\sigma)$, $G$ is a matrix and $\beta_{1},\beta_{2}$ are non-negative constants with $\beta_{1}+\beta_{2}>0$. 
Recently, Reisinge et al. \cite{reisinger2020path} proposed a generalized monotonicity condition by introducing two flexible measurable functions $\phi_{1}$ and $\phi_{2}$:
\begin{equation}\label{finite general monotonicity}
 \begin{aligned}
& \mathbb{E}\left[\left\langle \Gamma\left(t ; \Theta_{1}, \mathcal{L}(\Theta_{1})\right)-\Gamma\left(t ; \Theta_2, \mathcal{L}(\Theta_{2})\right), \Theta_1-\Theta_2\right\rangle \right] \\
& \leq-\beta_1 \phi_1\left(X_1, X_2\right)-\beta_2 \phi_2\left(t, \Theta_1, \Theta_2, \mathcal{L}(\Theta_{1}),\mathcal{L}(\Theta_{2})\right).
\end{aligned}    
\end{equation}
For the infinite horizon classical FBSDEs, Wei and Yu \cite{yu2021} proposed the following domination-monotonicity condition:
\begin{equation}\label{yu monotonicity}
\left\{
\begin{aligned}
&|g(s, x, y_{1}, z_{1})-g(s, x, y_{2}, z_{2})|   \leq\frac{1}{\nu}|A(s)(x_{1}-x_{2})|, \\
   & |F(s, x, y_{1}, z_{1})-F(s, x, y_{2}, z_{2})| \leq \frac{1}{\mu}|B(s) (y_{1}-y_{2})+C(s) (z_{1}-
{z}_{2})|, \\
  &\langle \Gamma(s, \theta_{1})-\Gamma(s, \theta_{2}), \theta_{1}-\theta_{2}\rangle+2 K\langle x_{1}-x_{2}, y_{1}-y_{2}\rangle \\
  & \quad  \leq-\nu |A(s)(x_{1}-x_{2})|^{2}-\mu|B(s) (y_{1}-y_{2})+C(s) (z_{1}-z_{2})|^2,
\end{aligned}\right.
\end{equation}
for any $\theta_{1} = (x_{1},y_{1},z_{1}),\ \theta_{2} = (x_{2},y_{2},z_{2})\in \mathbb{R}^{n+n+n \times d}$,  where $g = B,\sigma$, and $A(\cdot),B(\cdot),C(\cdot)$ are three bounded matrix-valued stochastic processes and $\nu,\mu$ are non-negative constants with $\nu+\mu>0$.
Motivated by above conditions, we introduce an infinite horizon generalized monotonicity condition involving two flexible functions $\phi_{1}$ and $\phi_{2}$ (see Assumption (H2)(i)). Moreover, different from finite horizon case, two additional monotonicity conditions for the coefficients $B$ and $F$ are proposed (see Assumption (H2)(ii)). 

Our work is closely related to \cite{bayraktar2023solvability} and \cite{yu2021}. We provide a comparison with their results and summarize our main innovations as follows:
 \begin{itemize}
     \item [(i)] We study a more general coupled (McKean-Vlasov) FBSDE in comparison with \cite{bayraktar2023solvability} and \cite{yu2021}.
     \item [(ii)] The introduction of the two functions $\phi_{1}$ and $\phi_{2}$ makes our conditions more general and flexible than those considered in \cite{bayraktar2023solvability} and \cite{yu2021}, thereby improving application scope of our solvability results. In particular, by selecting appropriate functions $\phi_{1}$ and $\phi_{2}$, our conditions can be reduced to those proposed in \cite{bayraktar2023solvability} and \cite{yu2021} (see Remark \ref{remark flexicity monotonicity} (ii)).
     \item [(iii)] As mentioned above, we need to identify the parameter $K$ based on the intrinsic properties of the coefficients, which is an additional consideration compared to finite horizon situation. When studying infinite horizon mean field control problems, under our generalized monotonicity condition, once the mean field control problem is well-defined for parameter $K$ determined by the coefficients, then it is solvable without additional constraint on $K$. In particular, we can solve the mean field control problem considered in \cite{bayraktar2023solvability} under a rather weaker constraint on $K$ (see Remark \ref{K}). Moreover, for stochastic LQ control problems on infinite horizon, \cite{yu2021} stated that there exists a phenomenon that whether the cross term coefficient $S(\cdot)$ in the cost functional is equal to zero or not may bring different results to the solvalibility of the LQ problem. With help of our theoretical result of FBSDE, we can eliminate this phenomenon under strictly convex condition (see Remark \ref{S}).
 \end{itemize}

The rest of the paper is organized as follows. In section 2, we present the necessary notations, concepts and study the well-posedness of the infinite horizon McKean-Vlasov SDEs and McKean-Vlasov BSDEs as a basis of the following study. In section 3, under the infinite horizon generalized monotonicity condition, we obtain the existence, uniqueness and related estimates of the solutions for the infinite horizon McKean-Vlasov FBSDEs with the method of continuation. Finally, in section 4, we investigate an infinite horizon mean field control problem by applying the obtained FBSDE result. 

\section{Preliminaries}
 Let $(\Omega, \mathcal{F}, \mathbb{P})$ be a complete probability space on which is defined a $d$-dimensional Brownian motion $W_{t}$, and $\mathbb{F}=\left\{\mathcal{F}_t\right\}_{t \geq 0}$ be the natural filtration of $W$ augmented with an independent $\sigma$-algebra $\mathcal{F}_{0}$. For any given $n,m \in \mathbb{N}$ and $x \in \mathbb{R}^n$, we denote by $\mathbb{I}_n$ the $n \times n$ identity matrix, by $\mathbb{R}^{n \times m}$ the Euclidean space of all $(n \times m)$ real matrices, especially, $\mathbb{R}^n=\mathbb{R}^{n \times 1}$, by $0_n$ the zero element of $\mathbb{R}^n$ and by $\delta_x$ the Dirac measure supported at $x$. We denote $\langle\cdot, \cdot\rangle$ and $|\cdot|$ to be the respective usual inner product and norm in Euclidean space, and for any $A,B \in \mathbb{R}^{n\times m}$, we define $\left \langle A,B \right \rangle  \triangleq \text{tr}(A^{\top}B)$, $|A|= \{\text{tr}(A^{\top}A)\}^{\frac{1}{2}}$, where the superscript $\top$ denotes the transpose of a vector or matrix. In this paper we use the operator norm of matrices:
 $$
\|A\|:=\sup _{0 \neq x \in \mathbb{R}^m} \frac{|A x|}{|x|}, \quad \text { for any } A \in \mathbb{R}^{n\times m} \text {. }
$$
Now we introduce some spaces which will be used in our following analysis.
For any $t\in[0, \infty)$ and constant $K \in \mathbb{R}$,
\begin{itemize}
\item $L_{\mathcal{F}_t}^2\left(\Omega ; \mathbb{R}^n\right)$ is the set of $\mathbb{R}^n$-valued $\mathcal{F}_t$-measurable random variables $\xi$ such that
$$
\|\xi\|_{L^2}:=\mathbb{E}\left[|\xi|^2\right]^{\frac{1}{2}}<\infty;
$$
    \item $L^{2,K}_{\mathbb{F}}(t,\infty;\mathbb{R}^{n})$ is that set of $\mathbb{R}^{n}$-valued $\mathbb{F}$-progressively measurable processes $f(\cdot)$ such that 
   \begin{equation}\label{space}
\|f(\cdot)\|_K:=\mathbb{E}\left[\int_t^{\infty} \left|e^{K s}f(s)\right|^2 d s\right] ^{\frac{1}{2}}< \infty;       
\end{equation}
\item  $L^{\infty}\left(t, \infty ; \mathbb{R}^{n \times m}\right)$ is the set of all Lebesgue measurable functions $A:[t, \infty) \rightarrow$ $\mathbb{R}^{n \times m}$ such that
$$
\|A(\cdot)\|_{\infty}:=\operatorname{esssup}_{s \in[t, \infty)}\|A(s)\|<\infty.
$$
\end{itemize}
Clearly, for any $K_1<K_2$, we have $L_{\mathbb{F}}^{2, K_2}\left(t, \infty ; \mathbb{R}^n\right) \subset L_{\mathbb{F}}^{2, K_1}\left(t, \infty ; \mathbb{R}^n\right)$, i.e., the sequence of spaces $\left\{L_{\mathbb{F}}^{2, K}\left(t, \infty ; \mathbb{R}^n\right)\right\}_{K \in \mathbb{R}}$ is decreasing in $K$.

In the sequel, we will use the notation $\mathcal{L}(\Theta)$ to denote the law of the random variable $\Theta$. Let $\mathcal{W}_2$ denote 2-Wassertein's distance on $\mathcal{P}_2\left(\mathbb{R}^n\right)$ defined by
\begin{equation}
\mathcal{W}_2\left(\mu_1, \mu_2\right) \triangleq \inf \left\{\left[\int_{\mathbb{R}^n \times \mathbb{R}^n}|x-y|^2 \pi(d x, d y)\right]^{\frac{1}{2}}, \pi \in \mathcal{P}_2\left(\mathbb{R}^n \times \mathbb{R}^n\right) \text { with marginals } \mu_1 \text { and } \mu_2\right\}.    
\end{equation}
 It is obvious from its definition that
$$
\mathcal{W}_{2}\left(\mu_1, \mu_2\right) \leq\mathbb{E}\left[\left|X_1-X_2\right|^2\right]^{\frac{1}{2}},
$$
where $X_1$ and $X_2$ are $n$-dimensional random vectors that follow the distributions $\mu_1$ and $\mu_2$ respectively.

For a function defined on space of measures, its Lipschitz continuity and differentiability upon the measure variable $\mu$ is understood in the sense of $2$-Wassertein distance and L-differentiability, respectively.

Now we briefly introduce the structure of L-derivative for functions defined on space of measures and we refer the readers to \cite[Chapter 5]{meanfield2018book1} for details. Let $\tilde{\Omega}$ be a Polish space and $\tilde{\mathbb{P}}$ an atomless measure over $\tilde{\Omega}$. 
 The notion of differentiability is based on the lifting of functions $\mathcal{P}_2\left(\mathbb{R}^d\right) \ni \mu \mapsto H(\mu)$ into functions $\tilde{H}$ defined on the Hilbert space $L^2(\tilde{\Omega} ; \mathbb{R}^d)$ over some probability space $(\tilde{\Omega}, \tilde{\mathcal{F}}, \tilde{\mathbb{P}})$ by setting $\tilde{H}(\tilde{X})=H(\tilde{\mathbb{P}}_{\tilde{X}})$ for $\tilde{X} \in L^2(\tilde{\Omega} ; \mathbb{R}^d)$. Then a function $H$ is said to be differentiable at $\mu_0 \in \mathcal{P}_2(\mathbb{R}^d)$ if there exists a random variable $\tilde{X}_0$ with law $\mu_0$ such that the lifted function $\tilde{H}$ is Fr\'{e}chet differentiable at $\tilde{X}_0$. Whenever this is the case, the Fr\'{e}chet derivative of $\tilde{H}$ at $\tilde{X}_0$ can be viewed as an element of $L^2(\tilde{\Omega} ; \mathbb{R}^d)$, denoted by $D\tilde{H}(\tilde{X}_{0})$, by identifying $L^2(\tilde{\Omega} ; \mathbb{R}^d)$ and its dual. It can be shown that there exists a measurable function : $\partial_\mu H(\mu_0): \mathbb{R}^d \rightarrow \mathbb{R}^d$ such that $\partial_\mu H(\mu_0)(\tilde{X}_0)=D \tilde{H}(\tilde{X}_0),\ \mathbb{P}$-a.s. Therefore, we define the derivative of $H$ at $\mu_0$ as the 
measurable function $\partial_\mu H(\mu_0)$, which satisfies
$$
H(\mu)=H\left(\mu_0\right)+\tilde{\mathbb{E}}\left[\partial_\mu H\left(\mu_0\right)(\tilde{X}_0) \cdot(\tilde{X}-\tilde{X}_0)\right]+o(\|\tilde{X}-\tilde{X}_0\|_2),
$$
where $\mathcal{L}(\tilde{X})=\mu, \mathcal{L}(\tilde{X}_0)=\mu_0$.

As a basis of the investigation of infinite horizon McKean-Vlasov FBSDEs, we give the well-posedness of infinite horizon McKean-Vlasov SDEs and infinite horizon McKean-Vlasov BSDEs in the rest of this section. The corresponding classical infinite SDEs and BSDEs have been studied in \cite{yu2021} and it can be observed that our conditions can degenerate to their conditions when there are no mean field terms.
\subsection{Infinite horizon McKean-Vlasov SDEs}
For $t\in[0, \infty)$, consider the following infinite horizon McKean-Vlasov SDE:
\begin{equation}\label{mean field sde}
\left\{  \begin{aligned}
        dX_{s} & = b\left(s,X_{s},\mathcal{L}(X_{s})\right)  ds+\sigma\left(t,X_{s},\mathcal{L}(X_{s})\right) dW_{s},\quad  s\in [t,\infty),\\
        X_{t} & = x_{t},    \end{aligned}\right.
\end{equation}
where $b: [t,\infty)\times \Omega \times \mathbb{R}^{n} \times \mathcal{P}_{2}(\mathbb{R}^{n}) \rightarrow \mathbb{R}^{n}$ and $\sigma:[t,\infty)\times \Omega \times \mathbb{R}^{n} \times \mathcal{P}_{2}(\mathbb{R}^{n})\rightarrow \mathbb{R}^{n\times d}$ are measurable functions. We introduce the following assumptions.
\begin{assumption}\label{assumption x}
(i) $x_{t}\in L^{2}_{\mathcal{F}_{t}}(\Omega;\mathbb{R}^{n})$. For any $x\in \mathbb{R}^{n},\ \mu\in \mathcal{P}_{2}(\mathbb{R}^{n})$, the processes $b(\cdot,x,\mu)$ and $\sigma(\cdot,x,\mu)$ are $\mathbb{F}$-progressively measurable. Moreover, there exists a constant $K \in \mathbb{R}$ such that $b(\cdot, 0,\delta_{0_{n}}) \in L_{\mathbb{F}}^{2, K}\left(t, \infty ; \mathbb{R}^n\right)$ and $\sigma(\cdot, 0,\delta_{0_{n}}) \in L_{\mathbb{F}}^{2, K}\left(t, \infty ; \mathbb{R}^{n \times d}\right)$.\\
(ii) The functions $b(t, x, \mu)$, $\sigma(t,x,\mu)$ are uniformly Lipschitz in $(x, \mu)$, i.e., there exist positive constants $l_{bx},\ l_{b\mu},\ l_{\sigma x},\ l_{\sigma \mu}$ such that for any $x,x^{\prime} \in \mathbb{R}^{n}$, $\mu, \mu^{\prime} \in \mathcal{P}_2(\mathbb{R}^{n})$ and almost all $(s,\omega) \in [t,\infty)\times \Omega$,  
\begin{equation}
\begin{aligned}
     |b(s, x,\mu)-b(s, x^{\prime},\mu^{\prime})| & \leq l_{bx}|x-x^{\prime}|+l_{b\mu} \mathcal{W}_2\left(\mu, \mu^{\prime}\right),\\
     |\sigma(s, x,\mu)-\sigma(s, x^{\prime},\mu^{\prime})| &\leq l_{\sigma x}|x-x^{\prime}|+ l_{\sigma\mu} \mathcal{W}_2\left(\mu, \mu^{\prime}\right).
\end{aligned}
\end{equation}
 (iii) There exists a constant $\kappa_{x}\in \mathbb{R}$ such that for any $\mu \in \mathcal{P}_2(\mathbb{R}^{n})$, $x, x^{\prime} \in \mathbb{R}^{n}$ and almost all $(s,\omega)\in [t,\infty)\times \Omega$, it holds that
\begin{equation}\label{monotonicity x}
\left \langle
x-x^{\prime},b(s, x, \mu)-b\left(s, x^{\prime}, \mu\right) \right \rangle \leq-\kappa_{x}\left|x-x^{\prime}\right|^2.    
\end{equation}
\end{assumption}
It follows from the classical theory of McKean-Vlasov SDEs on finite horizon (see \cite{meanfield2018book1}), under Assumption \ref{assumption x} (i) (ii), McKean-Vlasov SDE \eqref{mean field sde} admits a unique solution on $[t,\infty)$. Furthermore, similar with the proof of \cite[Proposition 2.1]{yu2017}), we can easily get the following result.
\begin{lemma}\label{infinite value}
     Let Assumption \ref{assumption x} (i) (ii) hold. If the solution $X$ to McKean-Vlasov SDE \eqref{mean field sde} belongs to $L_{\mathbb{F}}^{2, K}\left(t, \infty ; \mathbb{R}^n\right)$, then we have
\begin{equation}
\lim _{T \rightarrow \infty} \mathbb{E}\left[\left|e^{K T}X_{T}\right|^2\right]=0.  \end{equation}
\end{lemma}
Now, we give the main result for McKean-Vlasov SDE \eqref{mean field sde} as follows.
\begin{lemma}
    \label{x propostion}
     Let Assumption \ref{assumption x} holds. We further assume that $K<\kappa_{x}-\frac{(l_{\sigma x}+l_{\sigma \mu})^{2}}{2}-l_{b\mu}$. Then the solution $X$ to McKean-Vlasov SDE \eqref{mean field sde} belongs to $L_{\mathbb{F}}^{2, K}\left(t, \infty ; \mathbb{R}^n\right)$. Moreover, for any $\varepsilon>0$, we have the following estimate:
\begin{equation}\label{x estimate} 
\begin{aligned}
& \left(2 \kappa_{x}-2 K-2l_{b\mu}-(l_{\sigma x}+l_{\sigma \mu})^{2}-3 \varepsilon\right) \mathbb{E} \int_t^{\infty}\left|e^{K s}X_{s}\right|^2 d s \\
& \quad \leq \mathbb{E}\left\{\left|e^{K t}x_t\right|^2+\int_t^{\infty}\left[\frac{1}{\varepsilon}\left|e^{K s}b(s, 0,\delta_{0_{n}})\right|^2+\left(1+\frac{l_{\sigma x}^2+l_{\sigma \mu}^{2}}{\varepsilon}\right)\left|e^{K s}\sigma(s, 0,\delta_{0_{n}})\right|^2\right] d s\right\} .
\end{aligned}  
\end{equation} 
\end{lemma}

\begin{proof}
   For any $T>t$, applying It\^{o}'s formula to $\left|X_{s} e^{Ks}\right|^2$ on the interval $[t, T]$ yields
$$
\begin{aligned}
&\quad \mathbb{E}\left\{\left|X_{T}e^{K T}\right|^2-2 K \int_t^T\left|X_{s} e^{K s}\right|^2 d s\right\} \\
&=\mathbb{E}\left\{\left|x_t e^{K t}\right|^2+\int_t^T\left[2\langle X_{s}, b(s, X_{s},\mathcal{L}(X_{s}))\rangle+\left|\sigma\left(s, X_{s},\mathcal{L}(X_{s})\right)\right|^2\right] e^{2 K s} d s\right\} \\
&\leq \mathbb{E}\Bigg\{\left|x_t e^{K t}\right|^2+\int_t^T\Big[2\left\langle X_{s}, b(s, X_{s},\mathcal{L}(X_{s}))-b(s,0,\mathcal{L}(X_{s}))+b(s,0,\mathcal{L}(X_{s}))-b(s, 0,\delta_{0_{n}})\right\rangle\\
&\quad \quad +2|X_{s}||b(s, 0,\delta_{0_{n}})|+\left(|\sigma(s, X_{s},\mathcal{L}(X_{s}))-\sigma(s, 0,\delta_{0_{n}})|+|\sigma(s, 0,\delta_{0_{n}})|\right)^2\Big] e^{2 K s} d s\Bigg\} .
\end{aligned}
$$
By the monotonicity condition of $b$ and Lipschitz condition of $b$ and $\sigma$, we have
$$
\begin{aligned}
& \quad \mathbb{E}\left\{\left|X_{T} e^{K T}\right|^2-2 K \int_t^T\left|X_{s} e^{K s}\right|^2 d s\right\} \\
&  \leq \mathbb{E}\left\{\left|x_t e^{K t}\right|^2+\int_t^T\left[-2 \kappa_{x}|X_{s}|^2+2l_{b\mu}|X_{s}|\sqrt{\mathbb{E}[|X_{s}|^{2}]}+2|X_{s}||b(s, 0,\delta_{0_{n}})|\right.\right.\\
&\quad \quad  \left.\left.+\left(l_{\sigma x}|X_{s}|+l_{\sigma \mu}\sqrt{\mathbb{E}[|X_{s}|^{2}]}+|\sigma(s, 0,\delta_{0_{n}})|\right)^2\right] e^{2 K s} d s\right\}.
\end{aligned}
$$
For any $\varepsilon>0$, with the help of the inequality $2 a b \leq \varepsilon|a|^2+(1 / \varepsilon)|b|^2$, we derive
$$
\begin{aligned}
&\quad \mathbb{E}\left\{\left|X_{T} e^{K T}\right|^2-2 K \int_t^T\left|X_{s} e^{K s}\right|^2 d s\right\} \\
&\leq \mathbb{E}\left\{\left|x_t e^{K t}\right|^2+\int_t^T\left[\left(l_{\sigma x}^2+l_{\sigma \mu}^{2}+2l_{b\mu}+2l_{\sigma x}l_{\sigma\mu}+3 \varepsilon-2 \kappa_{x} \right)|X_{s}|^2+\frac{1}{\varepsilon}|b(s, 0,\delta_{0_{n}})|^2\right.\right. \\
& \quad \left.\left.\quad+\left(1+\frac{l_{\sigma x}^2+l_{\sigma\mu}^{2}}{\varepsilon}\right)|\sigma(s, 0,\delta_{0_{n}})|^2\right] e^{2 K s} d s\right\} .
\end{aligned}
$$
Let $T\rightarrow \infty$, thanks to Lemma \ref{infinite value}, we have
\begin{equation}\label{x estimate 2}
\begin{aligned}
& \left(2 \kappa_{x}-2 K-2l_{b\mu}-(l_{\sigma x}+l_{\sigma \mu})^{2}-3 \varepsilon\right) \mathbb{E} \int_t^\infty\left|X_{s} e^{K s}\right|^2 d s \\
& \quad \leq \mathbb{E}\left\{\left|x_t e^{K t}\right|^2+\int_t^\infty\left[\frac{1}{\varepsilon}\left|b(s, 0,\delta_{0_{n}}) e^{K s}\right|^2+\left(1+\frac{l_{\sigma x}^2+l_{\sigma \mu}^{2}}{\varepsilon}\right)\left|\sigma(s, 0,\delta_{0_{n}}) e^{K s}\right|^2\right] d s\right\} .
\end{aligned}    
\end{equation}
The condition $K<\kappa-\frac{(l_{\sigma x}+l_{\sigma \mu})^{2}}{2}-l_{b\mu}$ implies the existence of  the number $\varepsilon$ such that $3\varepsilon \in\left(0, \kappa-\frac{(l_{\sigma x}+l_{\sigma \mu})^{2}}{2}-l_{b\mu}-K\right)$. Therefore, the estimate \eqref{x estimate 2} and Assumption \ref{assumption x} (i) imply that the solution $X \in L_{\mathbb{F}}^{2, K}\left(t, \infty ; \mathbb{R}^n\right)$.
\end{proof}
\subsection{Infinite horizon McKean-Vlasov BSDEs}
For $t\in[0, \infty)$, we introduce the following infinite horizon McKean-Vlasov BSDE:
\begin{equation}\label{infinite bsde}
    \begin{aligned}
        dY_{s} = f(s,Y_{s},Z_{s},\mathcal{L}(Y_{s}), \mathcal{L}(Z_{s}))ds+Z_{s}dW_{s}, \quad s\in [t,\infty),
    \end{aligned}
\end{equation}
where $f:[t,\infty)\times \Omega \times \mathbb{R}^{m}\times \mathbb{R}^{m \times d}\times \mathcal{P}_{2}(\mathbb{R}^{m})\times \mathcal{P}_{2}(\mathbb{R}^{m \times d})\rightarrow\mathbb{R}^{m}$ is a measurable function and satisfies the following assumption.
\begin{assumption}\label{assumption y}
    (i)  For any $y\in \mathbb{R}^{m}$, $z\in \mathbb{R}^{m \times d}$, $\mu\in \mathcal{P}_{2}(\mathbb{R}^{m})$, $\nu \in \mathcal{P}_{2}( \mathbb{R}^{m \times d})$, the process $f(\cdot,y,z,\mu,\nu)$ is $\mathbb{F}$-progressively measurable. Moreover, there exists a constant $K \in \mathbb{R}$ such that $f(\cdot, 0,0,\delta_{0_{m}},\delta_{0_{m\times d}}) \in L_{\mathbb{F}}^{2, K}\left(t, \infty ; \mathbb{R}^m\right)$.\\
(ii) $f(s,y,z,\mu,\nu)$ is Lipschitz in $(y,z,\mu,\nu)$, i.e., there exist positive constants $l_{y},l_{z},l_{\mu_{y}},l_{\mu_{z}}$ such that for any  $y,y^{\prime} \in \mathbb{R}^{m}$, $z,z^{\prime} \in \mathbb{R}^{m \times d}$, $\mu, \mu^{\prime} \in \mathcal{P}_2(\mathbb{R}^{m})$, $\nu,\nu^{\prime} \in \mathcal{P}_{2}(\mathbb{R}^{m \times d})$ and almost all $(s,\omega) \in [t,\infty)\times \Omega$,  
\begin{equation}
     |f(s, y,z,\mu,\nu)-f(s, y^{\prime},z^{\prime},\mu^{\prime},\nu^{\prime})| \leq l_{y}|y-y^{\prime}|+l_{z}|z-z^{\prime}|+l_{\mu_{y}}\mathcal{W}_{2}(\mu,\mu^{\prime})+l_{\mu_{z}}\mathcal{W}_{2}(\nu,\nu^{\prime}).
\end{equation}
 (iii) There exists a constant $\kappa_{y}\in \mathbb{R}$ such that for any $y, y^{\prime} \in \mathbb{R}^{m}$, $z\in \mathbb{R}^{m \times d}$, $\mu \in \mathcal{P}_2(\mathbb{R}^{m})$, $\nu \in \mathcal{P}_{2}(\mathbb{R}^{m \times d})$ and almost all $(s,\omega)\in [t,\infty)\times \Omega$, it holds that
\begin{equation}\label{monotonicity y}
\left \langle
y-y^{\prime},f(s, y,z, \mu, \nu)-f\left(s, y^{\prime},z, \mu, \nu
\right) \right \rangle \geq-\kappa_{y}\left|y-y^{\prime}\right|^2.    
\end{equation}
\end{assumption}
First, we give the following an a priori estimate.
\begin{lemma}
    Let $f$ be a coefficient satisfying Assumption \ref{assumption y} and $K>\kappa_y+l_{\mu_{y}}+l_{z}^{2}+l_{\mu_{z}}^{2}$. Let $(Y, Z) \in L_{\mathbb{F}}^{2, K}\left(t, \infty ; \mathbb{R}^{m+m\times d}\right) $ be a solution to McKean-Vlasov BSDE \eqref{infinite bsde}. Then, for any $\varepsilon>0$, we have
\begin{equation}\label{BSDE estimate}
\begin{aligned}
& \mathbb{E}\left\{\left|Y_{t}e^{K t}\right|^2+\int_t^{\infty}\left[\left(2 K-2 \kappa_y-2l_{z}^2-2l_{\mu_{z}}^{2}-2
l_{\mu_{y}}-3 \varepsilon\right)\left|Y_{s} e^{K s}\right|^2\right.\right. \\
& \left.\left.\quad+\frac{\varepsilon}{l_{z}^2+l_{\mu_{z}}^{2}+\varepsilon}\left|Z_{s} e^{K s}\right|^2\right] d s\right\} \leq \frac{1}{\varepsilon} \mathbb{E} \int_t^{\infty}\left|f(s, 0,0,\delta_{0_{m}},\delta_{0_{m \times d}}) e^{K s}\right|^2 d s .
\end{aligned}    
\end{equation}
\end{lemma}
\begin{proof}
     For any $T \in(t, \infty)$, applying It\^{o}'s formula to $\left|Y_{s} e^{Ks}\right|^2$ on the interval $[t, T]$ leads to
\begin{equation}
\begin{aligned}
&\mathbb{E}\left\{\left|e^{K t}Y_{t} \right|^2+\int_t^T\left[2 K\left| e^{K s}Y_{s}\right|^2+\left| e^{K s}Z_{s}\right|^2\right] d s\right\} \\
&=\mathbb{E}\Bigg\{\left|e^{K T}Y_{T} \right|^2-2 \int_t^T\Big[\Big\langle Y_{s}, f(s, Y_{s}, Z_{s},\mathcal{L}(Y_{s}),\mathcal{L}(Z_{s}))-f(s, 0,Z_{s},\mathcal{L}(Y_{s}),\mathcal{L}(Z_{s}))\Big\rangle \\
&\quad \quad+\Big\langle Y_{s}, f(s, 0,Z_{s},\mathcal{L}(Y_{s}),\mathcal{L}(Z_{s}))-f(s, 0,0,\mathcal{L}(Y_{s}),\mathcal{L}(Z_{s}))\Big\rangle\\
&\quad \quad  +\Big\langle Y_{s}, f(s, 0,0,\mathcal{L}(Y_{s}),\mathcal{L}(Z_{s}))- f(s, 0,0,\delta_{0_{m}},\mathcal{L}(Z_{s}))\Big\rangle\\
&\quad \quad +\Big\langle Y_{s}, f(s, 0,0,\delta_{0_{m}},\mathcal{L}(Z_{s}))- f(s, 0,0,\delta_{0_{m}},\delta_{0_{m \times d}})\Big\rangle \Big] e^{2 K s} d s\Bigg\} \\
&\leq \mathbb{E}\Bigg\{\left|e^{K T}Y_{T}\right|^2+\int_t^T\Big[2 \kappa|Y_{s}|^2+2 l_{z}|Y_{s}||Z_{s}|+2l_{\mu_{y}}|Y_{s}|\sqrt{\mathbb{E}[|Y_{s}|^{2}]}+2l_{\mu_{z}}|Y_{s}|\sqrt{\mathbb{E}[|Z_{s}|^{2}]}\\
&\quad \quad  +2|Y_{s}||f(s,0,0,\delta_{0_{m}},\delta_{0_{m \times d}})|\Big] e^{2 K s} d s\Bigg\},
\end{aligned}    
\end{equation}
where the monotonicity condition and the Lipschitz condition of $f$ are used.\\
For any $\varepsilon>0$, by the inequalities
$$
2 a|y||z| \leq\left(l_{z}^2+l_{\mu_{z}}^{2}+\varepsilon\right)|y|^2+\frac{a^2}{l_{z}^2+l_{\mu_{z}}^{2}+\varepsilon}|z|^2, \quad 2|y||f| \leq \varepsilon|y|^2+\frac{1}{\varepsilon}|f|^2,
$$
we deduce that
$$
\begin{aligned}
& \mathbb{E}\left\{\left|Y_{t} e^{K t}\right|^2+\int_t^T\left[\left(2 K-2 \kappa_y-2l_{z}^2-2l_{\mu_{z}}^{2}-2
l_{\mu_{y}}-3 \varepsilon\right)\left|Y_{s} e^{K s}\right|^2+\frac{\varepsilon}{l_{ z}^2+l_{\mu_{z}}^{2}+\varepsilon}\left|Z_{s} e^{K s}\right|^2\right] d s\right\} \\
& \leq \mathbb{E}\left\{\left|Y_{T}e^{K T}\right|^2+\int_t^T \frac{1}{\varepsilon}\left|f(s, 0,0,\delta_{0_{m}},\delta_{0_{m \times d}
}) e^{K s}\right|^2 d s\right\} .
\end{aligned}
$$
 Then, by letting $T \rightarrow \infty$ on both sides of the above inequality, with help of Lemma \ref{infinite value}, we obtain the estimate \eqref{BSDE estimate}.
\end{proof}

Similar to the definition of solutions for classical BSDEs on infinite horizon in \cite{shi2020}, a pair of processes $(Y,Z)\in L^{2,K}_{\mathbb{F}}(t,\infty;\mathbb{R}^{m+m\times d})$ is called a solution to McKean-Vlasov BSDE \eqref{infinite bsde} if and only if, for any $T \in (t,\infty)$, the pair of processes $(Y,Z)$ satisfies 
\begin{equation}
    Y_{s} = Y_{T}-\int_{s}^{T}f(r,Y_{r},Z_{r},\mathcal{L}(Y_{r}),\mathcal{L}(Z_{r}))dr -\int_{s}^{T}Z_{r}dW_{r}, \quad s\in [t,T].
\end{equation}

With the help of above a priori estimate \eqref{BSDE estimate}, the method of Peng and Shi \cite[Theorem 4]{peng2000} is still valid to yield the following result.
\begin{lemma}\label{bsde lemma}
  Let the coefficient $f$ satisfies Assumption \ref{assumption y} and let $K>\kappa_y+l_{\mu_{y}}+l_{z}^{2}+l_{\mu_{z}}^{2}$. Then, infinite horizon McKean-Vlasov BSDE \eqref{infinite bsde} admits a unique solution $(Y, Z) \in L_{\mathbb{F}}^{2, K}\left(t, \infty ; \mathbb{R}^{m+m\times d}\right)$.   
\end{lemma}
\section{Infinite horizon McKean-Vlasov FBSDEs}
In this section, we establish existence and uniqueness of solution to infinite horizon McKean-Vlasov FBSDE
\begin{equation}\label{infinite FBSDE}
\left\{\begin{aligned}
d X_t&=B\left(t, X_t, Y_t, Z_{t},\mathcal{L}\left(X_t, Y_t, Z_{t}\right)\right)  dt+\sigma\left(t, X_t, Y_t, Z_{t},\mathcal{L}\left(X_t, Y_t, Z_{t}\right)\right)  d W_t, \quad t\in [0,\infty), \\
d Y_t&=F\left(t, X_t, Y_t, Z_{t}, \mathcal{L}\left(X_t, Y_t, Z_{t}\right)\right)  dt+Z_t  d W_t, \quad t\in [0,\infty), \\
X_0&=\xi,
\end{aligned}\right.    
\end{equation}
where $\xi \in  L^{2}_{\mathcal{F}_{0}}(\Omega;\mathbb{R}^{n})$.
\begin{definition}
    A triple of $\mathbb{F}$-progressively measurable processes $(X, Y, Z) \in L^{2, K}_{\mathbb{F}}\left(0, \infty ; \mathbb{R}^{n+m+m \times d}\right)$ is called a solution of FBSDE \eqref{infinite FBSDE}, if \eqref{infinite FBSDE} is satisfied in the following sense: for any $T>0$,
$$
\begin{aligned}
X_t & =  \xi+\int_0^t B(s,  X_s, Y_s, Z_{s},\mathcal{L}\left(X_s, Y_s,Z_{s}\right))  d s +\int_0^t \sigma(s,  X_s, Y_s, Z_{s}, \mathcal{L}\left( X_s, Y_s, Z_{s}\right))  d W_s,\quad t\in[0,T], \\
Y_t &  =  Y_{T}-\int_t^T F(s, X_s, Y_s, Z_{s}, \mathcal{L}\left( X_s, Y_s, Z_{s}\right))  d s-\int_t^T Z_{s}  d W_s, \quad t \in[0, T] .
\end{aligned}
$$
\end{definition}
We introduce the following assumptions.
\begin{assumption*}[\textbf{H1}]\label{H1}
  $\xi \in L^{2}_{\mathcal{F}_{0}}(\Omega;\mathbb{R}^{n})$.  Let $(B,F,\sigma):[0, \infty) \times \Omega \times \mathbb{R}^n \times \mathbb{R}^m \times \mathbb{R}^{m \times d} \times \mathcal{P}_2\left(\mathbb{R}^{n+m+m \times d} \right) \rightarrow (\mathbb{R}^n,\mathbb{R}^{m},\mathbb{R}^{m \times d})$ be $\mathbb{F}$- progressively measurable functions satisfying:\\
   (i)There exists a positive constant $l$ such that for any $x, x^{\prime} \in \mathbb{R}^{n}$, $y, y^{\prime} \in \mathbb{R}^{m}$, $z,z^{\prime} \in \mathbb{R}^{m \times d}$, $m, m^{\prime} \in \mathcal{P}_2\left(\mathbb{R}^{n+m+m \times d} \right)$ and almost all $(t,\omega)\in [0,\infty)\times \Omega$, 
\begin{equation*}
|(B,F,\sigma)(t, x, y, z,m)- (B,F,\sigma)\left(t, x^{\prime}, y^{\prime}, z^{\prime},m^{\prime}\right)| \leq l\left(\left|x-x^{\prime}\right|+\left|y-y^{\prime}\right|+\left|z-z^{\prime}\right|+\mathcal{W}_2\left(m, m^{\prime}\right)\right).   
\end{equation*}
(ii) There exists a constant $K\in \mathbb{R}$ such that $(B,F,\sigma)(\cdot,0,0,0,\delta_{0_{n}},\delta_{0_{m}},\delta_{0_{m\times d}})\in L^{2, K}_{\mathbb{F}}\left(0,\infty;\mathbb{R}^{n+m+n\times d}\right)$.
\end{assumption*}
Besides Assumption (H1), we introduce the following monotonicity condition to deal with the coupling between the forward equation and the backward equation in \eqref{infinite FBSDE} on infinite horizon.
\begin{assumption*}[\textbf{H2}]\label{H2}
     There exist constants $\kappa_{x},\kappa_{y}\in\mathbb{R}$, $ \beta_1, \beta_2 \in[0, \infty)$, $l_{\sigma}, l_{\phi},l_{z} \in (0,\infty)$, $\gamma \in (0,1)$, $G \in \mathbb{R}^{m \times n}$ and measurable functions $\phi_1: [0,\infty)\times  L^2\left(\Omega ; \mathbb{R}^n\right)\times  L^2\left(\Omega ; \mathbb{R}^n\right) \rightarrow[0, \infty)$, $\phi_2:[0, \infty) \times L^2\left(\Omega ; \mathbb{R}^{n+m+m \times d} \right)\times  L^2\left(\Omega ; \mathbb{R}^{n+m+m \times d} \right) \times \mathcal{P}_{2}(\mathbb{R}^{n+m+m \times d}) \times \mathcal{P}_{2}(\mathbb{R}^{n+m+m \times d})\rightarrow[0, \infty)$ such that for all $t \in[0, \infty)$, $i \in\{1,2\}, \Theta_i:=$ $\left(X_i, Y_i, Z_i\right) \in L^2\left(\Omega ; \mathbb{R}^{n+m+m \times d} \right)$, we have as follows: 
\begin{itemize}
    \item [(i)] One of the following two monotonicity conditions holds:
\begin{equation}\label{monotonicity}
 \begin{aligned}
& \mathbb{E}\left[\left\langle B\left(t,  \Theta_1,  \mathcal{L}(\Theta_1)\right)-B\left(t,  \Theta_2,  \mathcal{L}(\Theta_2)\right), G^{\top}\left(Y_1-Y_2\right)\right\rangle\right] \\
& \quad+\mathbb{E}\left[\left\langle\sigma\left(t,  \Theta_1,  \mathcal{L}(\Theta_1)\right)-\sigma\left(t,  \Theta_2,  \mathcal{L}(\Theta_2)\right), G^{\top}\left(Z_1-Z_2\right)\right\rangle\right] \\
& \quad+\mathbb{E}\left[\left\langle F\left(t, \Theta_1,  \mathcal{L}(\Theta_1)\right)-F\left(t, \Theta_2,  \mathcal{L}(\Theta_2)\right), G\left(X_1-X_2\right)\right\rangle\right] \\
&\quad +(\kappa_{x}+\kappa_{y})\mathbb{E}[\langle X_{1}-X_{2},G^{\top}(Y_{1}-Y_{2}) \rangle]\\
&\leq-\beta_1 \phi_1\left(t,X_1, X_2\right)-\beta_2 \phi_2\left(t, \Theta_1, \Theta_2,\mathcal{L}(\Theta_{1}), \mathcal{L}(\Theta_{2})\right),
\end{aligned}   
\end{equation}
and
\begin{equation}\label{monotonicity2}
 \begin{aligned}
& \mathbb{E}\left[\left\langle B\left(t,  \Theta_1,  \mathcal{L}(\Theta_1)\right)-B\left(t,  \Theta_2,  \mathcal{L}(\Theta_2)\right), G^{\top}\left(Y_1-Y_2\right)\right\rangle\right] \\
& \quad+\mathbb{E}\left[\left\langle\sigma\left(t,  \Theta_1,  \mathcal{L}(\Theta_1)\right)-\sigma\left(t,  \Theta_2,  \mathcal{L}(\Theta_2)\right), G^{\top}\left(Z_1-Z_2\right)\right\rangle\right] \\
& \quad+\mathbb{E}\left[\left\langle F\left(t, \Theta_1,  \mathcal{L}(\Theta_1)\right)-F\left(t, \Theta_2,  \mathcal{L}(\Theta_2)\right), G\left(X_1-X_2\right)\right\rangle\right] \\
&\quad +(\kappa_{x}+\kappa_{y})\mathbb{E}[\langle X_{1}-X_{2},G^{\top}(Y_{1}-Y_{2}) \rangle]\\
&\geq \beta_1 \phi_1\left(t,X_1, X_2\right)+\beta_2 \phi_2\left(t, \Theta_1, \Theta_2,\mathcal{L}(\Theta_{1}), \mathcal{L}(\Theta_{2})\right).
\end{aligned}   
\end{equation}
\item [(ii)]
 One of the following two cases holds.\\
\textit{Case 1:} $\beta_2>0$ and for any $t \in[0, \infty)$,
\begin{equation}\label{xy monotonicity}
\begin{aligned}
&\mathbb{E}\left[|\sigma\left(t,  \Theta_1, \mathcal{L}(\Theta_1)\right)- \sigma\left(t,  \Theta_2,  \mathcal{L}(\Theta_2)\right)|^{2} \right]
\\
& \quad \leq l_{\sigma}\mathbb{E}[|X_1-X_2|^{2}]+l_{\phi} \phi_2\left(t, \Theta_1, \Theta_2,\mathcal{L}(\Theta_{1}), \mathcal{L}(\Theta_{2})\right),   \\
&\mathbb{E}\left[\left\langle 
  B(t,\Theta_{1}, \mathcal{L}(\Theta_1))-B(t,\Theta_{2}, \mathcal{L}(\Theta_2)),X_{1}-X_{2}
 \right\rangle\right]\\
 & \quad \leq -\kappa_{x}\mathbb{E}[|X_{1}-X_{2}|^{2}]+l_{\phi} \phi_2\left(t, \Theta_1, \Theta_2,\mathcal{L}(\Theta_{1}), \mathcal{L}(\Theta_{2})\right),\\
 & \mathbb{E}\left[\langle 
  F(t,\Theta_{1},\mathcal{L}(\Theta_{1}))-F(t,\Theta_{2},\mathcal{L}(\Theta_{2})),Y_{1}-Y_{2}\rangle\right] \\
&\quad \geq -(\kappa_{y}+\frac{l_{z}}{2})\mathbb{E}\left[|Y_{1}-Y_{2}|^{2}\right]-\frac{\gamma}{2}\mathbb{E}[|Z_{1}-Z_{2}|^{2}]\\
&\quad \quad-l_{\phi}\left(\mathbb{E}[|X_{1}-X_{2}|^{2}]+ \phi_2\left(t, \Theta_1, \Theta_2,\mathcal{L}(\Theta_{1}), \mathcal{L}(\Theta_{2})\right)\right).
\end{aligned}
\end{equation}
\textit{Case 2:} $ \beta_1>0$ and for any $t \in[0, \infty)$,
\begin{equation}\label{xy monotonicity 2}
    \begin{aligned}
    &\mathbb{E}\left[|\sigma\left(t,  \Theta_1, \mathcal{L}(\Theta_1)\right)- \sigma\left(t,  \Theta_2,  \mathcal{L}(\Theta_2)\right)|^{2}   \right]\\
    & \quad \leq l_{\sigma}\mathbb{E}[|X_1-X_2|^{2}]+l_{\phi}\mathbb{E}\left[|Y_{1}-Y_{2}|^{2}+|Z_{1}-Z_{2}|^{2}\right],   \\
&\mathbb{E}\left[\left\langle 
  B(t,\Theta_{1}, \mathcal{L}(\Theta_1))-B(t,\Theta_{2}, \mathcal{L}(\Theta_2)),X_{1}-X_{2}
 \right\rangle\right]\\
 &\quad \leq -\kappa_{x}\mathbb{E}[|X_{1}-X_{2}|^{2}]+l_{\phi}\mathbb{E}\left[|Y_{1}-Y_{2}|^{2}+|Z_{1}-Z_{2}|^{2}\right],\\
  & \mathbb{E}\left[\langle 
   F(t,\Theta_{1}
,\mathcal{L}(\Theta_{1}))-F(t,\Theta_{2}
,\mathcal{L}(\Theta_{2})),Y_{1}-Y_{2}
\rangle\right] \\
&\quad \geq -(\kappa_{y}+\frac{l_{z}}{2})\mathbb{E}\left[|Y_{1}-Y_{2}|^{2}\right]-\frac{\gamma}{2}\mathbb{E}[|Z_{1}-Z_{2}|^{2}]-l_{\phi}\phi_{1}(t,X_{1},X_{2}).
\end{aligned}
\end{equation}
\end{itemize}
\end{assumption*}
We give the following remark to explain the Assumption (H2) and compare it with the existing literature.
\begin{remark}\label{remark flexicity monotonicity}
\begin{itemize}  
  \item [(i)] Compared with the finite horizon case in \cite{reisinger2020path}, an additional term $ (\kappa_{x}+\kappa_{y})\mathbb{E}[\langle X_{1}-X_{2},G^{\top}(Y_{1}-Y_{2})\rangle]$ appears in the monotonicity condition \eqref{monotonicity}, which will suit for the later analysis of infinite horizon.
    \item [(ii)] Compared with the monotonicity conditions in the literature, we propose a generalized monotonicity condition \eqref{monotonicity} by introducing two functions $\phi_{1}$ and $\phi_{2}$. The introduction of these two functions provides us with more flexibility. Specially, when we choose  $n=m=d=1$, $\beta_{1}=0$, $\beta_{2}>0$, $\phi_2\left(t, \Theta_1, \Theta_2,\mathcal{L}(\Theta_{1}),\mathcal{L}(\Theta_{2})\right)=\left\|\left(Y_1-Y_2\right)\right\|_{L^2}^2+\left\|\left(X_1-X_2\right)\right\|_{L^2}^2$, the monotonicity condition \eqref{monotonicity} is reduced to the monotonicity condition in \cite{bayraktar2023solvability}.  Moreover, our condition can be seen as a generalization of the domination-monotonicity condition proposed in \cite{yu2021} by choosing  
    $$
		\phi_1\left(t,X_1, X_2\right) =\left\| A(t)(X_{1}-X_{2})\right\|_{L^2}^2,
		$$
		$$
		\phi_{2} \left(t, \Theta_1, \Theta_2,\mathcal{L}(\Theta_{1}),\mathcal{L}(\Theta_{2})\right)=\left\|B(t)(Y_{1}-Y_{2})+C(t)(Z_{1}-Z_{2})\right\|_{L^2}^2,
		$$
     where $A(\cdot),B(\cdot),C(\cdot)$ are three bounded matrix-valued stochastic processes. Using a more general function $\phi_{2}$ instead of a linear combination of $Y$ and $Z$, our solvability results can be easily applied to infinite horizon mean field control problems whose coefficients enjoy specific structural conditions by choosing 
$$
\phi_{2}\left(t, \Theta_1, \Theta_2,\mathcal{L}(\Theta_{1}),\mathcal{L}(\Theta_{2})\right) = \|\hat{\alpha}(t,X_{1},Y_{1},Z_{1},\mathcal{L}(X_{1},Y_{1},Z_{1}))-\hat{\alpha}(t,X_{2},Y_{2},Z_{2},\mathcal{L}(X_{2},Y_{2},Z_{2}))\|^{2}_{L^{2}},
$$
where $\hat{\alpha}$ is the optimal control, which will be discussed in detail in section \ref{section 3}. 
    \item [(iii)] The monotonicity conditions of coefficients $B$ and $F$  are necessary assumptions in infinite horizon situations, which is indicated by the analysis of infinite horizon McKean-Vlasov SDEs and infinite horizon McKean-Vlasov BSDEs in the previous section. From the subsequent analysis, we can observe that the constants $\kappa_{x}$, $\kappa_{y}$ will affect the values of the parameter $K$. Condition \eqref{xy monotonicity} avoids the influence of the monotonicity of function $\phi_{2}$ on the choice of constants $\kappa_{x}$ and $\kappa_{y}$, further the values of parameter $K$. The benefits of this technique will be demonstrated in investigating mean field control problem \eqref{cost functional}-\eqref{state dynamics} and LQ control problems (see Remark \ref{S}).
\end{itemize}
\end{remark}
Now, we give the main result of this section.
\begin{Theorem}\label{global theorem}
 Let Assumptions \textnormal{(H1)} and \textnormal{(H2)} hold. Let
\begin{equation}\label{kxkyk}
\kappa_{x}-\kappa_{y}>\max\{l_{\sigma},l_{z}\} \quad \text{and}\quad  K = \frac{\kappa_{x}+\kappa_{y}}{2}.
\end{equation}
Then, infinite horizon McKean-Vlasov FBSDE \eqref{infinite FBSDE} admits a unique solution $(X,Y,Z)\in L^{2,K}_{\mathbb{F}}(0,\infty;$ $\mathbb{R}^{n+m+m\times d})$. Moreover, we have the following estimate:
\begin{equation}\label{stability 1}
\begin{aligned}
&\quad \mathbb{E}\int_{0}^{\infty}|e^{Kt}X_{t}|^{2}dt+\mathbb{E}\int_{0}^{\infty}|e^{Kt}Y_{t}|^{2}dt+\mathbb{E}\int_{0}^{\infty}|e^{Kt}Z_{t}|^{2}dt\\
&\leq C \mathbb{E}\int_{0}^{T}|(B,F,\sigma)(t,0,0,0,\delta_{0_{n+m+m\times d}})e^{Kt}|^{2}dt,
\end{aligned}
\end{equation}
where $C>0$ is a constant depending on $|G|,\kappa_{x},\kappa_{y},l_{\phi},l_{z},l_{\sigma},l$ and $\beta_{1}$ or $\beta_{2}$. Furthermore, let $(\bar{X},\bar{Y},\bar{Z}) \in L^{2,K}_{\mathbb{F}}(0,\infty;\mathbb{R}^{n+m+m\times d} )$ be a solution to the FBSDE \eqref{infinite FBSDE} with another set of coefficients $(\bar{\xi},\bar{B},\bar{F},\bar{\sigma})$, then we have
\begin{equation}\label{stability 2}
\begin{aligned}
 &\quad \mathbb{E}\int_{0}^{\infty}|e^{Kt}(X_{t}-\bar{X}_{t})|^{2}dt+\mathbb{E}\int_{0}^{\infty}|e^{Kt}(Y_{t}-\bar{Y}_{t})|^{2}dt+\mathbb{E}\int_{0}^{\infty}|e^{Kt}(Z_{t}-\bar{Z}_{t})|^{2}dt\\
&\leq C \left\{\mathbb{E}[|\xi-\bar{\xi}|^{2}]+\mathbb{E}\int_{0}^{\infty}|B(t,\bar{\Theta}_{t})-\bar{B}(t,\bar{\Theta}_{t})e^{Kt}|^{2}dt+\mathbb{E}\int_{0}^{\infty}|F(t,\bar{\Theta}_{t})-\bar{F}(t,\bar{\Theta}_{t})e^{Kt}|^{2}dt \right.\\
&\quad \left.\quad+\mathbb{E}\int_{0}^{\infty}|\sigma(t,\bar{\Theta}_{t})-\bar{\sigma}(t,\bar{\Theta}_{t})e^{Kt}|^{2}dt\right\},     
\end{aligned}
\end{equation}
where $\bar{\Theta}_{t}  = (\bar{X}_{t},\bar{Y}_{t},\bar{Z}_{t},\mathcal{L}(\bar{X}_{t},\bar{Y}_{t},\bar{Z}_{t}))$ and $C$ is the same constant as in \eqref{stability 1}.
\end{Theorem}
\begin{remark}
    It is easy to verify that \eqref{kxkyk} is equivalent to 
\begin{equation}\label{equivalentkxkyk}
\kappa_{y}+\frac{l_{z}}{2}<K<\kappa_{x}-\frac{l_{\sigma}}{2} \quad \text{and} \quad K = \frac{\kappa_{x}+\kappa_{y}}{2}.
    \end{equation}
\end{remark}
We introduce a family of infinite horizon FBSDEs parameterized by $\lambda \in[0,1]$:
\begin{equation}\label{parameter FBSDE}
\left\{\begin{aligned}
d X_t^\lambda &=\Big[\lambda B\left(t, X_t^\lambda, Y_t^\lambda, Z_{t}^{\lambda},\mathcal{L}\left(X_t^\lambda, Y_t^\lambda, Z_{t}^{\lambda}\right)\right)-(1-\lambda)\kappa_{x}X_{t}^{\lambda}+\mathcal{I}^{B}_{t}\Big]  d t\\
& \quad +\Big[\lambda\sigma\left(t, X_t^\lambda, Y_t^\lambda, Z_{t}^{\lambda}, \mathcal{L}\left(X_t^\lambda, Y_t^\lambda, Z_{t}^{\lambda}\right)\right)+\mathcal{I}^{\sigma}_{t}\Big] d W_{t}, \\
d Y_t^\lambda &=\Big[\lambda F\left(t, X_t^\lambda, Y_t^\lambda, Z_{t}^{\lambda}, \mathcal{L}\left(X_t^\lambda, Y_t^\lambda, Z_{t}^{\lambda}\right)\right)-(1-\lambda)\kappa_{y} Y_{t}^{\lambda}+\mathcal{I}^{F}_{t}\Big]  d t+Z_t^\lambda  d W_{t}, \\
X_0^\lambda&=\xi, 
\end{aligned}\right.    
\end{equation}
where $\xi \in  L^{2}_{\mathcal{F}_{0}}(\Omega;\mathbb{R}^{n})$ and $(\mathcal{I}^{B},\mathcal{I}^{F},\mathcal{I}^{\sigma})$ are arbitrary processes in $L_{\mathbb{F}}^{2,K}(0, \infty; \mathbb{R}^{n+m+n\times d})$. Note that when $\lambda=1$, $ \mathcal{I}^{B} \equiv 0$, $ \mathcal{I}^{F} \equiv 0$, $ \mathcal{I}^{\sigma} \equiv 0$, \eqref{parameter FBSDE} becomes \eqref{infinite FBSDE}, and when $\lambda=0$, FBSDE \eqref{parameter FBSDE} is reduced to 
\begin{equation}\label{initial FBSDE}
  \left\{\begin{aligned}
      d X_t^0 &=(-\kappa_{x} X_{t}^{0}+\mathcal{I}^{B}_{t})  d t+\mathcal{I}^{\sigma}_{t}  d W_t, \\
d Y_t^0&=(-\kappa_{y} Y_{t}^{0}+\mathcal{I}^{F}_{t}) d t+Z_t^0   d W_t, \\
X_0^0&=\xi.
  \end{aligned}
\right.  
\end{equation}
 It is clear that FBSDE \eqref{initial FBSDE} is in a decoupled form and we can solve the SDE and BSDE separately. As a direct application of Lemma \ref{x propostion} and Lemma \ref{bsde lemma}, we have the following result.
\begin{lemma}\label{initial lemma}
     Assume $\kappa_{y}<K<\kappa_{x}$, then for any $(\mathcal{I}^{B},\mathcal{I}^{F},\mathcal{I}^{\sigma}) \in L^{2, K}_{\mathbb{F}}\left(0, \infty ; \mathbb{R}^{n+m+n\times d}\right)$ and $\xi \in L^{2}_{\mathcal{F}_{0}}(\Omega;\mathbb{R}^{n})$, FBSDE \eqref{initial FBSDE} admits a unique solution $(X^{0},Y^{0}, Z^{0})$ in $L^{2, K}_{\mathbb{F}}\left(0, \infty ; \mathbb{R}^{n+m+m \times d}\right)$.
\end{lemma}
Next, for any $\lambda_{0} \in[0,1]$, we shall establish an a priori estimate for FBSDE \eqref{parameter FBSDE} which plays a key role in the method of continuation.
\begin{lemma}\label{stability}
Let $(\xi,B,F,\sigma)$, $(\bar{\xi},\bar{B},\bar{F},\bar{\sigma})$ satisfy Assumptions \textnormal{(H1)} and \textnormal{(H2)}, and $\left(\mathcal{I}^B, \mathcal{I}^F, \mathcal{I}^{\sigma}\right)$, $(\bar{\mathcal{I}}^B, \bar{\mathcal{I}}^{F}, \bar{\mathcal{I}}^{\sigma}) \in L^{2,K}_{\mathbb{F}}\left(0,\infty;\mathbb{R}^{n+m+n\times d}\right)$ and let
\begin{equation}
\kappa_{x}-\kappa_{y}>\max\{l_{\sigma},l_{z}\} \quad \text{and} \quad  K = \frac{\kappa_{x}+\kappa_{y}}{2}.
\end{equation}
Suppose  $(X,Y,Z),(\bar{X},\bar{Y},\bar{Z}) \in L^{2,K}_{\mathbb{F}}(0,\infty;\mathbb{R}^{n+m+m \times d})$ are solutions to FBSDE \eqref{parameter FBSDE} parameterized by $\lambda_{0}\in[0,1]$ with $\left(\xi,B,F,\sigma,\mathcal{I}^B, \mathcal{I}^F, \mathcal{I}^{\sigma}\right)$ and $(\bar{\xi},\bar{B},\bar{F},\bar{\sigma},\bar{\mathcal{I}}^B, \bar{\mathcal{I}}^{F}, \bar{\mathcal{I}}^{\sigma})$, respectively. Then  there exists a constant $C>0$ only depending on $|G|, \kappa_{x},\kappa_{y}$, $l_{\phi},l_{z},l_{\sigma}$ and $\beta_{1}$ or $\beta_{2}$,  independent of $\lambda_{0}$ such that 
\begin{equation}\label{stability equation}
\begin{aligned}
&\quad \left\| X -\bar{X}\right\|_{K}^{2}+\left\| Y-\bar{Y}\right\|_{K}^{2}+\left\| Z-\bar{Z} \right\|_{K}^{2}\\
& \leq  C \Big\{\mathbb{E}\left[|\xi-\bar{\xi}|^{2}\right]+\left\|\lambda_{0} \left(B(\cdot,\bar{\Theta})-\bar{B}(\cdot,\bar{\Theta})\right)+ (\mathcal{I}^B-\bar{\mathcal{I}}^{B}) \right\|_{K}^{2} \\
& \quad +\left\|\lambda_0 \left(\sigma(\cdot,\bar{\Theta})-\bar{\sigma}(\cdot,\bar{\Theta})\right)+(\mathcal{I}^{\sigma}-\bar{\mathcal{I}}^{\sigma})\right\|_{K}^{2} +\left\|\lambda_{0} \left(F(\cdot,\bar{\Theta})-\bar{F}(\cdot,\bar{\Theta})\right)+ (\mathcal{I}^F-\bar{\mathcal{I}}^{F}) \right\|_{K}^{2} \Big\},
\end{aligned}    
\end{equation}
where $\Theta = (X,Y,Z,\mathcal{L}(X,Y,Z))$, $\bar{\Theta} =(\bar{X},\bar{Y},\bar{Z},\mathcal{L}(\bar{X},\bar{Y},\bar{Z}))$ and $\|\cdot\|_{K}^{2}$ is defined as  \eqref{space}.
\end{lemma}
\begin{proof}
The whole proof will be splitted into two cases according to Assumption (H2). Before splitting the proof, we first do some pretreatments for both cases.\\
First, we can represent the FBSDE \eqref{parameter FBSDE} parameterized by $\lambda_{0}$ in the following form:
    \begin{equation}
        \begin{aligned}
        d e^{Kt}X_t & =\left[\lambda_{0} e^{Kt} B\left(t, X_t, Y_t, Z_{t},\mathcal{L}\left(X_t, Y_t, Z_{t}\right)\right)+(K-(1-\lambda_{0})\kappa_{x})e^{Kt}X_{t}+e^{Kt}\mathcal{I}^{B}_{t}\right] dt\\
        &\quad + \left[\lambda_{0} e^{Kt} \sigma\left(t, X_t, Y_t, Z_{t}, \mathcal{L}\left(X_t, Y_t, Z_{t}\right)\right)+e^{Kt}\mathcal{I}_t^\sigma\right]d W_t\\
        de^{Kt} Y_t & = \left[\lambda_{0} e^{Kt}F\left(t, X_t, Y_t, Z_{t},\mathcal{L}\left(X_t, Y_t, Z_{t}\right)\right)+(K-(1-\lambda_{0})\kappa_{y})e^{Kt}Y_{t}+e^{Kt}\mathcal{I}^{F}_{t}\right] dt+e^{Kt}Z_{t}dW_{t}.
         \end{aligned}
    \end{equation}
By applying  It\^{o}'s formula to $\left\langle e^{Kt}(Y_t-\bar{Y}_t), Ge^{Kt}\left(X_t-\bar{X}_t\right)\right\rangle$ on the time interval $[0,T]$, we obtain that
\begin{equation}
\begin{aligned}
& \quad\mathbb{E}\left[\left\langle e^{KT}(Y_{T}-\bar{Y}_{T}), Ge^{KT}(X_{T}-\bar{X}_{T})\right\rangle\right]-\mathbb{E}\left[\left\langle Y_{0}-\bar{Y}_{0},G(\xi-\bar{\xi})\right\rangle\right] \\
& =\mathbb{E}\int_0^T \Bigg[\big\langle\lambda_{0}e^{Kt}\left(B(t,\Theta_{t})-\bar{B}(t,\bar{\Theta}_{t})\right)+(K-(1-\lambda_{0})\kappa_{x})e^{Kt}(X_{t}-\bar{X}_{t})+e^{Kt}(\mathcal{I}^{B}_{t}-\bar{\mathcal{I}}^{B}_{t}),G^{\top}e^{Kt}(Y_{t}-\bar{Y}_{t})\big\rangle\\
&\quad \quad +\left\langle\lambda_{0}e^{Kt}\left(F(t,\Theta_{t})-\bar{F}(t,\bar{\Theta}_{t})\right)+(K-(1-\lambda_{0})\kappa_{y})e^{Kt}(Y_{t}-\bar{Y}_{t})+e^{Kt}(\mathcal{I}^{F}_{t}-\bar{\mathcal{I}}^{F}_{t}),Ge^{Kt}(X_{t}-\bar{X}_{t})\right\rangle\\
&\quad \quad +\left\langle\lambda_{0}e^{Kt}\left(\sigma(t,\Theta_{t})-\bar{\sigma}(t,\bar{\Theta}_{t})\right)+e^{Kt}(\mathcal{I}^{\sigma}_{t}-\bar{\mathcal{I}}^{\sigma}_{t}),G^{\top}e^{Kt}(Z_{t}-\bar{Z}_{t})\right\rangle \Bigg]
dt.
\end{aligned}    
\end{equation}
Then, by adding and subtracting the terms $ B\left(t,\bar{\Theta}_t\right)$, $ F\left(t,\bar{\Theta}_t\right)$, $ \sigma\left(t,\bar{\Theta}_t\right)$, we can deduce that 
\begin{equation}
\begin{aligned}
& \quad \mathbb{E}\left[\left\langle e^{KT}(Y_{T}-\bar{Y}_{T}), Ge^{KT}(X_{T}-\bar{X}_{T})\right\rangle\right]-\mathbb{E}\left[\left\langle Y_{0}-\bar{Y}_{0},G(\xi-\bar{\xi})\right\rangle\right]\\
 & =  \mathbb{E}\int_{0}^{T}\Bigg[\big\langle\lambda_{0}e^{Kt}\left(B(t,\Theta_{t})-B(t,\bar{\Theta}_{t})\right),G^{\top}e^{Kt}(Y_{t}-\bar{Y}_{t})\big\rangle\\
 &\quad \quad \quad+\big\langle\lambda_{0}e^{Kt}\left(F(t,\Theta_{t})-F(t,\bar{\Theta}_{t})\right),Ge^{Kt}(X_{t}-\bar{X}_{t})\big\rangle \\
&\quad \quad\quad +\left\langle\lambda_0 e^{Kt}\left(\sigma(t,\Theta_t)-\sigma(t,\bar{\Theta}_t)\right), G^{\top}e^{Kt}\left(Z_t-\bar{Z}_t\right)\right\rangle\\
&\quad \quad\quad+\big(2K-(1-\lambda_{0})(\kappa_{x}+\kappa_{y})\big)\left\langle e^{Kt}\left(X_{t}-\bar{X}_{t}\right),G^{\top}e^{Kt}\left(Y_{t}-\bar{Y}_{t}\right) \right\rangle\\
&\quad \quad\quad+ \left\langle\lambda_{0}e^{Kt}\left(B(t,\bar{\Theta}_{t})-\bar{B}(t,\bar{\Theta}_{t})\right)+e^{Kt}\left(\mathcal{I}^{B}_{t}-\bar{\mathcal{I}}^{B}_{t}\right),G^{\top}e^{Kt}\left(Y_{t}-\bar{Y}_{t}\right)\right\rangle\\
&\quad \quad\quad+\left\langle\lambda_{0}e^{Kt}\left(F(t,\bar{\Theta}_{t})-\bar{F}(t,\bar{\Theta}_{t})\right)+e^{Kt}\left(\mathcal{I}^{F}_{t}-\bar{\mathcal{I}}^{F}_{t}\right),Ge^{Kt}\left(X_{t}-\bar{X}_{t}\right)\right\rangle\\
&\quad\quad \quad+\left\langle\lambda_0 e^{Kt}\left(\sigma(t,\bar{\Theta}_t)-\bar{\sigma}(t,\bar{\Theta}_t)\right)+e^{Kt}( \mathcal{I}_t^{\sigma}-\bar{\mathcal{I}}_{t}^{\sigma}), G^{\top}e^{Kt}\left(Z_t-\bar{Z}_t\right)\right\rangle \Bigg] dt.
\end{aligned}
\end{equation}
The monotonicity condition \eqref{monotonicity} and $K = (\kappa_{x}+\kappa_{y})/2$ work together to reduce the above equation to 
\begin{equation}
\begin{aligned}
& \quad \mathbb{E}\left[\left\langle e^{KT}(Y_{T}-\bar{Y}_{T}), Ge^{KT}(X_{T}-\bar{X}_{T})\right\rangle\right]-\mathbb{E}\left[\left\langle Y_{0}-\bar{Y}_{0},G(\xi-\bar{\xi})\right\rangle\right]\\
 & \leq \mathbb{E}\int_{0}^{T}\Bigg[\left\langle\lambda_{0}e^{Kt}\left(B(t,\bar{\Theta}_{t})-\bar{B}(t,\bar{\Theta}_{t})\right)+e^{Kt}(\mathcal{I}^{B}_{t}-\bar{\mathcal{I}}^{B}_{t}),G^{\top}e^{Kt}(Y_{t}-\bar{Y}_{t})\right\rangle\\
&\quad \quad +\left\langle\lambda_{0}e^{Kt}\left(F(t,\bar{\Theta}_{t})-\bar{F}(t,\bar{\Theta}_{t})\right)+e^{Kt}(\mathcal{I}^{F}_{t}-\bar{\mathcal{I}}^{F}_{t}),Ge^{Kt}(X_{t}-\bar{X}_{t})\right\rangle\\
&\quad \quad +\left\langle\lambda_0 e^{Kt}\left(\sigma\left(t,\bar{\Theta}_t\right)-\bar{\sigma}\left(t,\bar{\Theta}_t\right)\right)+e^{Kt}( \mathcal{I}_t^{\sigma}-\bar{\mathcal{I}}_{t}^{\sigma}), G^{\top}e^{Kt}\left(Z_t-\bar{Z}_t\right)\right\rangle \Bigg]dt\\
&\quad -\lambda_0 \int_0^Te^{2Kt}\left(\beta_1\phi_1\left(t,X_t, \bar{X}_t\right)+\beta_2 \phi_2\left(t, \Theta_t, \bar{\Theta}_t\right)\right) d t.
\end{aligned}
\end{equation}
With help of the inequality $2ab \leq \varepsilon a^{2}+ (1/\varepsilon) b^{2}$ for any $\varepsilon>0$, Lemma \ref{infinite value}, and letting $T\rightarrow \infty$, we have 
\begin{equation}\label{general inequality}
 \begin{aligned}
& \quad\lambda_0 \int_0^{\infty}e^{2Kt}\left(\beta_1 \phi_1\left(t,X_t, \bar{X}_t\right)+\beta_2 \phi_2\left(t, \Theta_t, \bar{\Theta}_t\right)\right) d t \\
& \leq   \varepsilon\left(\left\|Y_0-\bar{Y}_0\right\|_{L^2}^2+\|X-\bar{X}\|_{K}^{2}+\|Y-\bar{Y}\|_{K}^{2}+\|Z-\bar{Z}\|_{K}^{2}\right)\\
&  \quad+\frac{|G|^{2}}{4\varepsilon}\Big\{\mathbb{E}\left[|\xi-\bar{\xi}|^{2}\right]+\left\|\lambda_{0} \left(B\left(\cdot,\bar{\Theta}\right)-\bar{B}\left(\cdot,\bar{\Theta}\right)\right)+ (\mathcal{I}^B-\bar{\mathcal{I}}^{B}) \right\|_{K}^{2} \\
&\quad +\left\|\lambda_0 \left(\sigma\left(\cdot,\bar{\Theta}\right)-\bar{\sigma}\left(\cdot,\bar{\Theta}\right)\right)+(\mathcal{I}^{\sigma}-\bar{\mathcal{I}}^{\sigma})\right\|_{K}^{2} +\left\|\lambda_{0} \left(F\left(\cdot,\bar{\Theta}\right)-\bar{F}\left(\cdot,\bar{\Theta}\right)\right)+ (\mathcal{I}^F-\bar{\mathcal{I}}^{F}) \right\|_{K}^{2} \Big\}.
\end{aligned}   
\end{equation}
Moreover, under monotonicity condition $\eqref{monotonicity2}$, we can still get \eqref{general inequality} with similar arguments. We have finished the pretreatment work. The next analysis in two cases will be based on \eqref{general inequality}. For simplicity of notations, from now we denote the right-hand side of \eqref{stability equation} as $\mathrm{RHS}$.

\textbf{Case 1}: $\beta_2>0$.
First, applying  It\^{o}'s formula to $|e^{Kt}(X_{t}-\bar{X}_{t})|^{2}$ on the time interval $[0,T]$, we have
\begin{equation}
\begin{aligned}
    &\quad \mathbb{E}\left[|e^{KT}(X_{T}-\bar{X}_{T})|^{2}\right]-\mathbb{E}\left[|\xi-\bar{\xi}|^{2}\right]\\
    & =\mathbb{E}\int_{0}^{T}\Bigg[2\left\langle e^{Kt}(X_{t}-\bar{X}_{t}),\lambda_{0}e^{Kt}\left(B(t,\Theta_{t})-B(t,\bar{\Theta}_{t})+B(t,\bar{\Theta}_{t})-\bar{B}(t,\bar{\Theta}_{t})\right)+e^{Kt}(\mathcal{I}^{B}_{t}-\bar{\mathcal{I}}^{B}_{t})\right. \\
    & \quad \left. \quad +\left(K-(1-\lambda_{0})\kappa_{x}\right)e^{Kt}(X_{t}-\bar{X}_{t}) \right\rangle+ \left|\lambda_{0}e^{Kt}\big(\sigma(t,\Theta_{t})-\sigma(t,\bar{\Theta}_{t}) +\sigma(t,\bar{\Theta}_{t})-\bar{\sigma}(t,\bar{\Theta}_{t})\big)\right.\\
    &\quad \left.\quad+e^{Kt}(\mathcal{I}^{\sigma}_{t}-\bar{\mathcal{I}}^{\sigma}_{t})\right|^{2}\Bigg]d t.
\end{aligned}
\end{equation}
Under the condition \eqref{xy monotonicity} and with the help of the inequality $2ab\leq \varepsilon_{1} a^{2} +(1/\varepsilon_{1})b^{2}$ for any $\varepsilon_{1}>0$, we can deduce that, 
\begin{equation}
    \begin{aligned}
&\mathbb{E}\left[|e^{KT}(X_{T}-\bar{X}_{T})|^{2}\right]-\mathbb{E}\left[|\xi-\bar{\xi}|^{2}\right]\\
&\leq \mathbb{E}\int_{0}^{T}\Bigg[
(2K-2\kappa_{x})e^{2Kt}|X_{t}-\bar{X}_{t}|^{2}+2\lambda_{0}l_{\phi}e^{2Kt}\phi_{2}(t,\Theta_{t},\bar{\Theta}_{t})+\varepsilon_{1} |e^{Kt}(X_{t}-\bar{X}_{t})|^{2}\\
&\quad \quad +\frac{1}{\varepsilon_{1}}|\lambda_{0}e^{Kt}\left(B(t,\bar{\Theta}_{t})-\bar{B}(t,\bar{\Theta}_{t})\right)+e^{Kt}(\mathcal{I}^{B}_{t}-\bar{\mathcal{I}}^{B}_{t})|^{2}+\lambda_{0}^{2}\left(l_{\sigma}|e^{Kt}(X_{t}-\bar{X}_{t})|^{2}+l_{\phi}e^{2Kt}\phi_{2}(t,\Theta_{t},\bar{\Theta}_{t})\right)\\
&\quad \quad  +|\lambda_{0}e^{Kt}\left(\sigma(t,\bar{\Theta}_{t})-\bar{\sigma}(t,\bar{\Theta}_{t})\right)+e^{Kt}(\mathcal{I}^{\sigma}_{t}-\bar{\mathcal{I}}^{\sigma}_{t})|^{2}+\varepsilon_{1} \lambda_{0}^{2}\left(l_{\sigma}|e^{Kt}(X_{t}-\bar{X}_{t})|^{2}+l_{\phi}e^{2Kt}\phi_{2}(t,\Theta_{t},\bar{\Theta}_{t})\right)\\
&\quad  \quad +\frac{1}{\varepsilon_{1}}|\lambda_{0}e^{Kt}\left(\sigma(t,\bar{\Theta}_{t})-\bar{\sigma}(t,\bar{\Theta}_{t})\right)+e^{Kt}(\mathcal{I}^{\sigma}_{t}-\bar{\mathcal{I}}^{\sigma}_{t})|^{2}\Bigg]d t
.
\end{aligned}
\end{equation}
Arrange terms, we can obtain that
\begin{equation}\label{case1x}
    \begin{aligned}
       & \quad \mathbb{E}[|e^{KT}(X_{T}-\bar{X}_{T})|^{2}]+\left(2\kappa_{x}-2K-l_{\sigma}-(1+l_{\sigma})\varepsilon_{1}\right)\mathbb{E}\int_{0}^{T}|e^{Kt}(X_{t}-\bar{X}_{t})|^{2}d t\\
       &\leq (2+\varepsilon_{1})l_{\phi}\lambda_{0}\int_{0}^{T}e^{2Kt}\phi_{2}(t,\Theta_{t},\bar{\Theta}_{t})d t+\frac{1}{\varepsilon_{1}}\mathbb{E}\int_{0}^{T}e^{2Kt}\left|(\lambda_{0} \left(B(t,\bar{\Theta}_{t})-\bar{B}(t,\bar{\Theta}_{t})\right)+\mathcal{I}^{B}_{t}-\bar{\mathcal{I}}^{B}_{t}\right|^{2}d t\\
       &\quad +\left(\frac{1}{\varepsilon_{1}}+1\right)\mathbb{E}\int_{0}^{T}e^{2Kt}\left|\lambda_{0}\left(\sigma(\bar{\Theta}_{t})-\bar{\sigma}(t,\bar{\Theta}_{t})\right)+\mathcal{I}^{\sigma}_{t}-\bar{\mathcal{I}}^{\sigma}_{t}\right|^{2}d t+\mathbb{E}[|\xi-\bar{\xi}|^{2}].
    \end{aligned}
\end{equation}
The condition  $K< \kappa_{x}-(l_{\sigma}/2)$ implies the existence of $\varepsilon_{1}$ such that $(1+l_{\sigma})\varepsilon_{1} \in (0,\kappa_{x}-K-(l_{\sigma}/2))$ and we denote
\begin{equation}
C_{1}:=\frac{(2+\varepsilon_{1})l_{\phi}}{(2\kappa_{x}-2K-l_{\sigma}-(1+l_{\phi})\varepsilon_{1})\beta_{2}},\quad C_{2}:=\frac{(1/\varepsilon_{1})+1}{2\kappa_{x}-2K-l_{\sigma}-(1+l_{\phi})\varepsilon_{1}}.
\end{equation}
When $\beta_{2}>0$, from \eqref{general inequality}, we obtain,
\begin{equation}\label{case1}
\begin{aligned}
  &\quad \lambda_0 \beta_{2} \int_0^{\infty} e^{2Kt}\phi_2\left(t, \Theta_t, \bar{\Theta}_t\right) d t\\
  &\leq \varepsilon\Big(\left\|Y_0-\bar{Y}_0\right\|_{L^2}^2+\|X-\bar{X}\|_{K}^{2}+\|Y-\bar{Y}\|_{K}^{2}+\|Z-\bar{Z}\|_{K}^{2}\Big)+\frac{|G|^{2}}{4\varepsilon} \mathrm{RHS}.
\end{aligned}    
\end{equation}
Combine \eqref{case1} and \eqref{case1x} and let $T\rightarrow \infty$, we have
\begin{equation}\label{xestimate1}
    \begin{aligned}
        &\quad \|X-\bar{X}\|_{K}^2
        \\
        &\leq C_{1}\left\{\varepsilon\Big(\left\|Y_0-\bar{Y}_0\right\|_{L^2}^2+\|X-\bar{X}\|_{K}^2+\|Y-\bar{Y}\|_{K}^2+\|Z-\bar{Z}\|_{K}^2\Big)
        +\frac{|G|^{2}}{4\varepsilon} \mathrm{RHS} \right\}+C_{2}\mathrm{RHS}.
    \end{aligned}
\end{equation}
Choose $\varepsilon$ small enough such that $0<\varepsilon < 1/C_{1}$, and denote
\begin{equation}
   \bar{\varepsilon}:= \frac{C_{1}\varepsilon}{1-C_{1}\varepsilon},\quad C_{3}:= \frac{C_{1}|G|^{2}}{4\varepsilon(1-C_{1}\varepsilon)}+\frac{C_{2}}{1-C_{1}\varepsilon}.
\end{equation}
Then, \eqref{xestimate1} is reduced to 
\begin{equation}\label{X}
    \|X-\bar{X}\|_{K}^2 \leq \bar{\varepsilon} \Big(\left\|Y_0-\bar{Y}_0\right\|_{L^2}^2+\|Y-\bar{Y}\|_{K}^2+\|Z-\bar{Z}\|_{K}^2\Big)+C_{3}\mathrm{RHS}.
\end{equation}
Second, we apply It\^{o}'s formula to $|e^{Kt}(Y_{t}-\bar{Y}_{t})|^{2}$  on the time interval $[0,T]$ and we obtain
\begin{equation}\label{ito furmula y}
    \begin{aligned}
        &\quad \mathbb{E}[|e^{KT}(Y_{T}-\bar{Y}_{T})|^{2}]-\mathbb{E}[|Y_{0}-\bar{Y}_{0}|^{2} ]\\
&=\mathbb{E}\int_{0}^{T}\Big[2\left\langle e^{Kt}(Y_{t}-\bar{Y}_{t}),\lambda_{0}e^{Kt}\left(F(t,\Theta_{t})-\bar{F}(t,\bar{\Theta}_{t})\right)+(K-(1-\lambda_{0})\kappa_{y})e^{Kt}(Y_{t}-\bar{Y}_{t})\right.\\
        &\quad \quad \left.+e^{Kt}(\mathcal{I}^{F}_{t}-\bar{\mathcal{I}}^{F}_{t}) \right \rangle \Big]d t+\mathbb{E}\int_{0}^{T}|e^{Kt}(Z_{t}-\bar{Z}_{t})|^{2}d t \\
        & = \mathbb{E}\int_{0}^{T}\Big[2\big\langle e^{Kt}(Y_{t}-\bar{Y}_{t}),\lambda_{0}e^{Kt}\left(F(t,\Theta_{t})-F(t,\bar{\Theta}_{t}) +F(t,\bar{\Theta}_{t})-\bar{F}(t,\bar{\Theta}_{t})\right)\\
        &\quad\quad +(K-(1-\lambda_{0})\kappa_{y})e^{Kt}(Y_{t}-\bar{Y}_{t})+e^{Kt}(\mathcal{I}^{F}_{t}-\bar{\mathcal{I}}^{F}_{t}) \big\rangle\Big] d t+\mathbb{E}\int_{0}^{T}|e^{Kt}(Z_{t}-\bar{Z}_{t})|^{2}d t. 
    \end{aligned}
\end{equation}
Under the condition \eqref{xy monotonicity} and with the inequality, for any $\varepsilon_{2}>0$, $ 2ab\geq -\varepsilon_{2} a^{2} -(1/\varepsilon_{2})b^{2}$, we can deduce that,
\begin{equation}\label{case1 y}
    \begin{aligned}
         &\quad \mathbb{E}[|e^{KT}(Y_{T}-\bar{Y}_{T})|^{2}]-\mathbb{E}[|Y_{0}-\bar{Y}_{0}|^{2} ]\\
         & \geq -2l_{\phi}\lambda_{0}\mathbb{E}\int_{0}^{T}e^{2Kt}|X_{t}-\bar{X}_{t}|^{2}dt-2l_{\phi}\lambda_{0}\int_{0}^{T}e^{2Kt}\phi_{2}(t,\Theta_{t},\bar{\Theta}_{t})dt\\
         &\quad +(2K-2\kappa_{y}-\lambda_{0}l_{z})\mathbb{E}\int_{0}^{T}|e^{Kt}(Y_{t}-\bar{Y}_{t})|^{2}dt -\lambda_{0}\gamma\mathbb{E}\int_{0}^{T}|e^{Kt}(Z_{t}-\bar{Z}_{t})|^{2}dt+\mathbb{E}\int_{0}^{T}|e^{Kt}(Z_{t}-\bar{Z}_{t})|^{2}dt\\
         &\quad -\varepsilon_{2}\mathbb{E}\int_{0}^{T}|e^{Kt}(Y_{t}-\bar{Y}_{t})|^{2}dt -\frac{1}{\varepsilon_{2}}\mathbb{E}\int_{0}^{T}\left|\lambda_{0}e^{Kt}\left(F(t,\bar{\Theta}_{t})-\bar{F}(t,\bar{\Theta}_{t})\right)+e^{Kt}(\mathcal{I}^{F}_{t}-\bar{\mathcal{I}}^{F}_{t})\right|^{2} dt.
         \end{aligned}
\end{equation}
Arranging terms, letting $T\rightarrow \infty$ and substituting \eqref{case1} into \eqref{case1 y} yield
\begin{equation}\label{case1 yy}
   \begin{aligned}
       &\quad \mathbb{E}\left[|Y_{0}-\bar{Y}_{0}|^{2}\right]+(1-\gamma)\mathbb{E}\int_{0}^{\infty}|e^{Kt}(Z_{t}-\bar{Z}_{t})|^{2}d t +\left(2K-2\kappa_{y}-l_{z}-\varepsilon_{2}\right)\mathbb{E}\int_{0}^{\infty}|e^{Kt}(Y_{t}-\bar{Y}_{t})|^{2}d t\\
       & \leq 2l_{\phi}\mathbb{E}\int_{0}^{\infty}e^{2Kt}|X_{t}-\bar{X}_{t}|^{2}dt+\frac{2l_{\phi}\varepsilon}{\beta_{2}}\Big(\left\|Y_0-\bar{Y}_0\right\|_{L^2}^2+\|X-\bar{X}\|_{K}^2+\|Y-\bar{Y}\|_{K}^2+\|Z-\bar{Z}\|_{K}^2\Big)\\
       &\quad+\frac{l_{\phi}|G|^{2}}{2\varepsilon \beta_{2}}\mathrm{RHS} +\frac{1}{\varepsilon_{2}}\mathbb{E}\int_{0}^{\infty}|\lambda_{0}e^{Kt}(F(t,\bar{\Theta}_{t})-\bar{F}(t,\bar{\Theta}_{t})+(\mathcal{I}^{F}_{t}-\bar{\mathcal{I}}^{F}_{t}))|^{2}d t
   \end{aligned}
\end{equation}
The condition $K>\kappa_{y}+(l_{z}/2)$ implies the existence of the number $\varepsilon_{2}$ such that $\varepsilon_{2}\in (0,K-\kappa_{y}-(l_{z}/2))$. We can choose $\varepsilon$ small enough such that 
\begin{equation}\label{varepsilon}
  2K-2\kappa_{y}-l_{z}-\varepsilon_{2}-\frac{2l_{\phi}\varepsilon}{\beta_{2}}>0, \quad 1-\gamma-\frac{2l_{\phi}\varepsilon}{\beta_{2}}>0.
\end{equation}
And we denote
\begin{equation}
C_{4}:=\operatorname{min}\left\{2K-2\kappa_{y}-l_{z}-\varepsilon_{2}-\frac{2l_{\phi}\varepsilon}{\beta_{2}},1-\gamma-\frac{2l_{\phi}\varepsilon}{\beta_{2}}\right\}.
\end{equation}
Therefore, \eqref{case1 yy} is reduced to
\begin{equation}\label{y0yz}
    \left\|Y_0-\bar{Y}_0\right\|_{L^2}^2+\|Y-\bar{Y}\|_{K}^2+\|Z-\bar{Z}\|_{K}^2 \leq C_{5}\|X-\bar{X}\|_{K}^2+C_{6}\mathrm{RHS},
\end{equation}
where
\begin{equation}
    C_{5}:=\frac{2l_{\phi}(1+\frac{\varepsilon}{\beta_{2}})}{C_{4}}, \quad C_{6}:=\frac{\frac{l_{\phi}^{2}|G|^{2}}{2\varepsilon\beta_{2}}+\frac{1}{\varepsilon_{2}}}{C_{4}}.
\end{equation}
Choose $\varepsilon$ small enough such that $1-\bar{\varepsilon}C_{5} >0$ and combined with \eqref{X}, we obtain
\begin{equation}\label{Y}
    \|X-\bar{X}\|_{K}^2 \leq \frac{C_{6}\bar{\varepsilon}+C_{3}}{1-\bar{\varepsilon}C_{5}}\mathrm{RHS}.
\end{equation}
Substituting \eqref{Y} into \eqref{y0yz}, we have
\begin{equation}
    \left\| X -\bar{X}\right\|_{K}^2+\left\| Y-\bar{Y}\right\|_{K}^2+\left\| Z-\bar{Z} \right\|_{K}^2 \leq C\mathrm{RHS},
\end{equation}
where 
\begin{equation}
    C := \frac{C_{6}\bar{\varepsilon}+C_{3}}{1-\bar{\varepsilon}C_{5}}+\frac{C_{5}(C_{6}\bar{\varepsilon}+C_{3})}{1-\bar{\varepsilon}C_{5}}+C_{6}.
\end{equation}
The desired estimate \eqref{stability equation} is obtained in case 1.

\textbf{Case 2:} $\beta_{1}>0$. When $\beta_{1}>0$, from \eqref{general inequality}, we obtain,
\begin{equation}\label{case2}
\begin{aligned}
    &\quad \lambda_0 \beta_{1} \int_0^\infty e^{2Kt}\phi_1\left(t, X_{1}, X_{2}\right) d t\\
 & \leq  \varepsilon\Big(\left\|Y_0-\bar{Y}_0\right\|_{L^2}^2+\|X-\bar{X}\|_{K}^{2}+\|Y-\bar{Y}\|_{K}^{2}+\|Z-\bar{Z}\|_{K}^{2}\Big)+\frac{|G|^{2}}{4\varepsilon} \mathrm{RHS}.
\end{aligned}    
\end{equation}
We first apply It\^{o}'s formula to $|e^{Kt}(Y_{t}-\bar{Y}_{t})|^{2}$ on time interval $[0,T]$. Under condition \eqref{xy monotonicity 2}, combined with \eqref{case2} and letting $T\rightarrow \infty$, \eqref{ito furmula y} is reduced to
\begin{equation}\label{case2 yy}
   \begin{aligned}
       &\quad \mathbb{E}\left[|Y_{0}-\bar{Y}_{0}|^{2}\right]+(1-\gamma)\mathbb{E}\int_{0}^{\infty}|e^{Kt}(Z_{t}-\bar{Z}_{t})|^{2}d t +\left(2K-2\kappa_{y}-l_{z}-\varepsilon_{2}\right)\mathbb{E}\int_{0}^{\infty}|e^{Kt}(Y_{t}-\bar{Y}_{t})|^{2}d t\\
       & \leq \frac{2l_{\phi}\varepsilon}{\beta_{1}}\Big(\left\|Y_0-\bar{Y}_0\right\|_{L^2}^2+\|X-\bar{X}\|_{K}^2+\|Y-\bar{Y}\|_{K}^2+\|Z-\bar{Z}\|_{K}^2\Big)\\
       &\quad+\frac{l_{\phi}|G|^{2}}{2\varepsilon \beta_{1}}\mathrm{RHS} +\frac{1}{\varepsilon_{2}}\mathbb{E}\int_{0}^{\infty}|\lambda_{0}e^{Kt}\left(F(t,\bar{\Theta}_{t})-\bar{F}(t,\bar{\Theta}_{t})\right)+(\mathcal{I}^{F}_{t}-\bar{\mathcal{I}}^{F}_{t})|^{2}d t.
   \end{aligned}
\end{equation}
 We can choose $\varepsilon$ small enough such that 
\begin{equation}\label{varepsilon 2}
  2K-2\kappa_{y}-l_{z}-\varepsilon_{2}-\frac{2l_{\phi}\varepsilon}{\beta_{1}}>0, \quad 1-\gamma-\frac{2l_{\phi}\varepsilon}{\beta_{1}}>0,
\end{equation}
and denote
\begin{equation}
C_{7}:=\operatorname{min}\left\{2K-2\kappa_{y}-l_{z}-\varepsilon_{2}-\frac{2l_{\phi}\varepsilon}{\beta_{1}},1-\gamma-\frac{2l_{\phi}\varepsilon}{\beta_{1}}\right\}.
\end{equation}
Then, we obtain 
\begin{equation}\label{case 2 yyz}
     \left\|Y_0-\bar{Y}_0\right\|_{L^2}^2+\|Y-\bar{Y}\|_{K}^2+\|Z-\bar{Z}\|_{K}^2 \leq \tilde{\varepsilon}\|X-\bar{X}\|_{K}^2+C_{8}\mathrm{RHS},
\end{equation}
where \begin{equation}
   \tilde{\varepsilon}:= \frac{2l_{\phi}\varepsilon}{\beta_{1}C_{7}},\quad C_{8}:= \frac{\frac{l_{\phi}^{2}|G|^{2}}{2\varepsilon\beta_{1}}+\frac{1}{\varepsilon_{2}}}{C_{7}}.
\end{equation}
Next, we apply It\^{o}'s formula to $|e^{Kt}(X_{t}-\bar{X}_{t})|^{2}$ on time interval $[0,T]$. Under condition \eqref{xy monotonicity 2}, by similar estimating argument with case 1, we can get
\begin{equation}\label{case2x}
    \begin{aligned}
       & \quad \mathbb{E}[|e^{KT}(X_{T}-\bar{X}_{T})|^{2}]+\left(2\kappa_{x}-2K-l_{\sigma}-(1+l_{\sigma})\varepsilon_{1}\right)\mathbb{E}\int_{0}^{T}|e^{Kt}(X_{t}-\bar{X}_{t})|^{2}d t\\
       &\leq \mathbb{E}[|\xi-\bar{\xi}|^{2}]+(2+\varepsilon_{1})l_{\phi}\int_{0}^{T}e^{2Kt}(|Y_{t}-\bar{Y}_{t}|^{2}+|Z_{t}-\bar{Z}_{t}|^{2})d t\\
       &\quad +\frac{1}{\varepsilon_{1}}\mathbb{E}\int_{0}^{T}e^{2Kt}\left|(\lambda_{0} \left(B(t,\bar{\Theta}_{t})-\bar{B}(t,\bar{\Theta}_{t})\right)+\mathcal{I}^{B}_{t}-\bar{\mathcal{I}}^{B}_{t}\right|^{2}d t\\
       &\quad +\left(\frac{1}{\varepsilon_{1}}+1\right)\mathbb{E}\int_{0}^{T}e^{2Kt}\left|\lambda_{0}\left(\sigma(t,\bar{\Theta}_{t})-\bar{\sigma}(t,\bar{\Theta}_{t})\right)+\mathcal{I}^{\sigma}_{t}-\bar{\mathcal{I}}^{\sigma}_{t}\right|^{2}d t.
    \end{aligned}
\end{equation}
Denote
\begin{equation}
C_{9}:=\frac{(2+\varepsilon_{1})l_{\phi}}{2\kappa_{x}-2K-l_{\sigma}-(1+l_{\phi})\varepsilon_{1}}. 
\end{equation}
Then, letting $T\rightarrow \infty$, \eqref{case2x} is reduce to 
\begin{equation}\label{case 2 xx}
        \|X-\bar{X}\|_{K}^2 \leq C_{9}(\|Y-\bar{Y}\|_{K}^2+\|Z-\bar{Z}\|_{K}^2)+C_{2}\mathrm{RHS}.
\end{equation}
Substituting \eqref{case 2 yyz} into \eqref{case 2 xx}, we have
\begin{equation}
     \|X-\bar{X}\|_{K}^2  \leq C_{9}(\tilde{\varepsilon}\|X-\bar{X}\|_{K}^2+C_{8}\mathrm{RHS})+C_{2}\mathrm{RHS}.
\end{equation}
By choosing $\varepsilon$ such that $1-C_{9}\tilde{\varepsilon}>0$, we have
\begin{equation}
    \|X-\bar{X}\|_{K}^2  \leq \frac{C_{9}C_{8}+C_{2}}{1-C_{9}\tilde{\varepsilon}} \mathrm{RHS}.
\end{equation}
Combined with \eqref{case 2 yyz}, we obtain
\begin{equation}
    \left\| X -\bar{X}\right\|_{K}^2+\left\| Y-\bar{Y}\right\|_{K}^2+\left\| Z-\bar{Z} \right\|_{K}^2 \leq \tilde{C}\mathrm{RHS}, 
\end{equation}
where 
$$
\tilde{C}:= \frac{(1+\tilde{\varepsilon})(C_{9}C_{8}+C_{2})}{1-C_{9}\tilde{\varepsilon}}+C_{8}.
$$
The desired estimates \eqref{stability equation} is obtained in case 2. Therefore, we obtain the estimate \eqref{stability equation} in two cases. The whole proof is completed.
\end{proof}

Based on the above a priori estimate, we give a continuation lemma.
\begin{lemma}\label{continuation lemma}
    Suppose Assumptions \textnormal{(H1)} and \textnormal{(H2)} hold and let
\begin{equation}
\kappa_{x}-\kappa_{y}>\operatorname{max}\{l_{\sigma},l_{z}\}\quad \text{ and }\quad  K = \frac{\kappa_{x}+\kappa_{y}}{2}.
\end{equation}
Then there exists a constant $\delta_{0}>0$ independent of $\lambda_{0}$ such that if for some $\lambda_{0} \in [0,1)$ and any $\left(\mathcal{I}^B, \mathcal{I}^F, \mathcal{I}^{\sigma}\right) \in L^{2,K}_{\mathbb{F}}\left(\mathbb{R}^{n+m+n\times d}\right)$, FBSDE \eqref{parameter FBSDE} admits a unique solution $(X,Y,Z) \in L^{2,K}_{\mathbb{F}}(0,\infty;$ $\mathbb{R}^{n+m+m\times d})$, then the same conclusion is also true for $\lambda= \lambda_{0}+\delta$ with $\delta \in [0,\delta_{0}]$ and $\lambda\leq 1$.
\end{lemma}
\begin{proof}
    Let $\delta_{0}>0$ be determined later, and $\delta \in [0,\delta_{0}]$. For any $\left(\mathcal{I}^B, \mathcal{I}^{F}, \mathcal{I}^{\sigma}\right) \in L^{2,K}_{\mathbb{F}}\left(0,\infty;\mathbb{R}^{n+m+n\times d} \right)$ and $\theta=(x,y,z) \in L^{2,K}_{\mathbb{F}}\left(0,\infty;\mathbb{R}^{n+m+m \times d} \right)$, we introduce an infinite horizon McKean-Vlasov
 FBSDE as follows:
\begin{equation}
\left\{\begin{aligned}
d X_t &=\Big[\lambda_{0} B\left(t, X_t, Y_t, Z_t, \mathcal{L}(X_t, Y_t, Z_t)\right)-(1-\lambda_{0})\kappa_{x} X_{t}+\tilde{\mathcal{I}}^{B}_{t}\Big] d t\\
& \quad +\Big[\lambda_{0}\sigma\left(t, X_t, Y_t, Z_t, \mathcal{L}\left(X_t, Y_t, Z_t\right)\right)+\tilde{\mathcal{I}}^{\sigma}_{t}\Big]d W_t, \\
d Y_t&=\Big[\lambda_{0}F\left(t, X_t, Y_t, Z_t, \mathcal{L}\left(X_t, Y_t,Z_{t}\right)\right)-(1-\lambda_{0})\kappa_{y} Y_{t}+\tilde{\mathcal{I}}^{F}_{t}\Big] d t+Z_t d W_t, \\
X_0&=\xi, 
\end{aligned}\right. \label{lambda0 FBSDE}  
\end{equation}
where 
\begin{equation}
\left\{
    \begin{aligned}
        \tilde{\mathcal{I}}^{B}_{t} &= \delta\big(B(t,x_{t},y_t,z_t,\mathcal{L}(x_{t},y_t,z_t))+\kappa_{x} x_t\big)+\mathcal{I}_{t}^{B},\\
         \tilde{\mathcal{I}}^{F}_{t} &= \delta\big(F(t,x_{t},y_t,z_t,\mathcal{L}(x_{t},y_t,z_t))+\kappa_{y} y_t\big)+\mathcal{I}_{t}^{F},\\
        \tilde{\mathcal{I}}^{\sigma}_{t} &= \delta\big(\sigma(t,x_{t},y_t,z_t,\mathcal{L}(x_{t},y_t,z_t)\big)+\mathcal{I}_{t}^{\sigma}.
    \end{aligned}\right.
    \end{equation}
    Since $(x,y,z)\in L_{\mathbb{F}}^{2,K}(0,\infty;\mathbb{R}^{n+m+m \times d})$ and $\left(\mathcal{I}^B, \mathcal{I}^F, \mathcal{I}^{\sigma}\right) \in L^{2,K}_{\mathbb{F}}\left(\mathbb{R}^{n+m+n\times d}\right)$, by Assumption (H1), it is easy to check that $(\tilde{\mathcal{I}}^{B}, \tilde{\mathcal{I}}^{F}, \tilde{\mathcal{I}}^{\sigma}) \in L^{2,K}_{\mathbb{F}}(0,\infty;\mathbb{R}^{n+m+n\times d})$. Then, by our assumptions, FBSDE \eqref{lambda0 FBSDE} admits a unique solution $\Theta=(X,Y,Z)\in L^{2,K}_{\mathbb{F}}(0,\infty;\mathbb{R}^{n+m+m \times d})$.
Due to the arbitrariness of $\theta$, we can comprehend that FBSDE \eqref{lambda0 FBSDE} defines a mapping from the Banach space $L^{2,K}_{\mathbb{F}}\left(0,\infty; \mathbb{R}^{n+m+m \times d}\right)$ into itself:
$$ \Theta=\mathcal{M}_{\lambda_0+\delta}( \theta).
$$
In the following, we shall prove that this mapping is contractive when $\delta$ is small enough.

For any $\theta=(x,y,z),\ \theta^{\prime}=(x^{\prime},y^{\prime},z^{\prime}) \in L^{2, K}_{\mathbb{F}}\left(0, \infty ; \mathbb{R}^{n+m+m \times d}\right)$, let $\Theta=(X,Y,Z) = \mathcal{M}_{\lambda_{0}+\delta}(\theta)$ and $\Theta^{\prime} = (X^{\prime},Y^{\prime},Z^{\prime}) = \mathcal{M}_{\lambda_{0}+\delta}(\theta^{\prime})$. By  Lemma \ref{stability}, there exists a constant $C>0$, independent of $\lambda_0$ such that
\begin{equation}
    \begin{aligned}
&\quad \left\| X  -X^{\prime}\right\|_{K}^2+\left\| Y-Y^{\prime}\right\|_{K}^2+\left\| Z-Z^{\prime} \right\|_{K}^2 \\
&\leq C\Big\{\left\|\delta\Big(B\left(\cdot,\theta, \mathcal{L}(\theta)\right)-B\left(\cdot, \theta^{\prime}, \mathcal{L}(\theta^{\prime})\right)+\kappa_{x}(x-x^{\prime})\Big)\right\|_{K}^2 \\
&\quad+\left\|\delta \Big(\sigma\left(\cdot,\theta, \mathcal{L}(\theta)\right)-\sigma\left(\cdot, \theta^{\prime}, \mathcal{L}(\theta^{\prime})\right)\Big)\right\|_{K}^2\\
&\quad+\left\|\delta\Big(F\left(\cdot,\theta, \mathcal{L}(\theta)\right)-F\left(\cdot, \theta^{\prime}, \mathcal{L}(\theta^{\prime})\right)+\kappa_{y}(y-y^{\prime})\Big)\right\|_{K}^2\Big\}.
\end{aligned}  
\end{equation}
Combined the fact that
\begin{equation}
\mathcal{W}_2\left(\mathcal{L}\left(x_t, y_t,z_{t}\right), \mathcal{L}\left(x_t^{\prime}, y_t^{\prime},z^{\prime}_{t}\right)\right) \leq \mathbb{E}\left[\left|x_t-x_t^{\prime}\right|^2\right]^{\frac{1}{2}}+\mathbb{E}\left[\left|y_t-y_t^{\prime}\right|^2\right]^{\frac{1}{2}}+\mathbb{E}\left[\left|z_t-z_t^{\prime}\right|^2\right]^{\frac{1}{2}},
\end{equation}
and  the Lipschitz continuity of functions $B,F,\sigma$, we have
\begin{equation}\label{contraction}
  \| X  -X^{\prime}\left\|_{K}^2+\right\| Y-Y^{\prime}\left\|_{K}^2+\right\| Z-Z^{\prime} \|_{K}^2 \leq 
   \bar{C} \delta^2\left( \| x  -x^{\prime}\left\|_{K}^2+\right\| y-y^{\prime}\left\|_{K}^2+\right\| z-z^{\prime} \|_{K}^2\right),
\end{equation}
    where $\bar{C}$ only depending on $|G|,\kappa_{x},\kappa_{y},l,l_{\sigma},l_{\phi},l_{z},\gamma, \beta_{1}$ or $\beta_{2}$ and independent of $\lambda_{0}$, which shows  we can choose $\delta_{0}$ independent of $\lambda_{0}$ such that $\bar{C}\delta^{2}<1$ when $\delta \leq \delta_{0}$. Thus $\mathcal{M}_{\lambda_{0}+\delta}$ is a contraction mapping and the fixed point is the unique solution of infinite horizon McKean-Vlasov FBSDE \eqref{parameter FBSDE} parameterized by $\lambda_{0}+\delta$.
\end{proof}

Now, we shall establish the well-posedness for infinite horizon FBSDE \eqref{infinite FBSDE} by applying Lemma \ref{initial lemma} and Lemma \ref{continuation lemma}.
\begin{proof}[Proof of Theorem \ref{global theorem}] 
When $\lambda = 0$, Lemma \ref{initial lemma} shows that FBSDE \eqref{parameter FBSDE}
is uniquely solvable in $L^{2,K}_{\mathbb{F}}(0,\infty;\mathbb{R}^{n+m+m \times d})$ for any
 $\left(\mathcal{I}^B, \mathcal{I}^{F}, \mathcal{I}^{\sigma}\right) \in L^{2,K}_{\mathbb{F}}\left(0,\infty;\mathbb{R}^{n+m+n\times d}\right)$. Based on this, Lemma \ref{continuation lemma} further implies that FBSDE \eqref{parameter FBSDE} is uniquely solvable for any $\lambda\in [0,1]$ and  any
 $\left(\mathcal{I}^B, \mathcal{I}^{F}, \mathcal{I}^{\sigma}\right) \in L^{2,K}_{\mathbb{F}}\left(0,\infty;\mathbb{R}^{n+m+n\times d}\right)$. Especially, when $\lambda = 1$ and $\left(\mathcal{I}^B, \mathcal{I}^{F}, \mathcal{I}^{\sigma}\right)\equiv 0$, \eqref{parameter FBSDE} coincides with the original FBSDE \eqref{infinite FBSDE}. Consequently, the unique solvability of \eqref{infinite FBSDE} is obtained. 

 The estimates $\eqref{stability 2}$ is followed from \eqref{stability equation} in Lemma \ref{stability} by letting $\lambda_{0} = 1$, $\left(\mathcal{I}^B, \mathcal{I}^{F}, \mathcal{I}^{\sigma}\right)\equiv 0$ and $\left(\bar{\mathcal{I}}^B, \bar{\mathcal{I}}^{F}, \bar{\mathcal{I}}^{\sigma}\right)\equiv 0$.

 Moreover, when $(\bar{B},\bar{F},\bar{\sigma}) \equiv 0$, $\xi = 0$, it is obvious $(0,0,0)$ is a solution to the corresponding FBSDE \eqref{infinite FBSDE}. Then we get the estimate \eqref{stability 1} from \eqref{stability 2}.
\end{proof}
\section{Infinite horizon mean field control problems}\label{section 3}
In this section, we investigate mean field control problems through the solvability results of infinite horizon McKean-Vlasov FBSDEs obtained in previous section. First, in subsection \ref{subsection 1}, we derive the corresponding infinite horizon FBSDE \eqref{mfc fbsde} by Pontryagin's stochastic maximum principle and solve the control problem given solutions to \eqref{mfc fbsde}. Then in subsection \ref{subsection 2}, we provide sufficient conditions for the existence of solutions to \eqref{mfc fbsde}. Let $A\in \mathbb{R}^{m}$ $(m\geq 1)$ be a convex control space. Suppose $b:[0,\infty)\times \mathbb{R}^{n}\times \mathcal{P}_{2}(\mathbb{R}^{n})\times A \rightarrow \mathbb{R}^{n}$, $\sigma:[0,\infty)\times \mathbb{R}^{n}\times \mathcal{P}_{2}(\mathbb{R}^{n})\times A \rightarrow \mathbb{R}^{n\times d}$, $f: [0,\infty)\times \mathbb{R}^{n}\times \mathcal{P}_{2}(\mathbb{R}^{n})\times A \rightarrow \mathbb{R}$ are three measurable functions. We work under the following assumption.
\begin{assumption}\label{assumption control x}
(i) $b(t, x, \mu, \alpha),\sigma(t,x,\mu,\alpha)$ are Lipschitz in $(x, \mu, \alpha)$ and $f(t,x,\mu,\alpha)$ is of at most quadratic growth in $(x,\mu,a)$. There exist positive constants $l_{bx},l_{b\mu},l_{\sigma x},l_{\sigma \mu},l_{\alpha}$ such that for any $t>0$, $\alpha,\alpha^{\prime} \in A$, $x,x^{\prime} \in \mathbb{R}^{n}$, $\mu, \mu^{\prime} \in \mathcal{P}_2(\mathbb{R}^{n})$,
\begin{equation}
\begin{aligned}
       |b(t, x,\mu,\alpha)-b(t, x^{\prime},\mu^{\prime},\alpha^{\prime})|  & \leq l_{bx}|x-x^{\prime}|+l_{\alpha}|\alpha-\alpha^{\prime}|+l_{b\mu} \mathcal{W}_2\left(\mu, \mu^{\prime}\right), \\
       |\sigma(t, x,\mu,\alpha)-\sigma(t, x^{\prime},\mu^{\prime},\alpha^{\prime})| &\leq l_{\sigma x}|x-x^{\prime}|+l_{\alpha}|\alpha-\alpha^{\prime}|
+l_{\sigma\mu} \mathcal{W}_2\left(\mu, \mu^{\prime}\right).
    \end{aligned}
\end{equation}
(ii) There exists a constant $K \in \mathbb{R}$ such that $\int_0^{\infty} e^{2Kt}\left|f\left(t, 0, \delta_{0_{n}}, 0\right)\right| d t<+\infty$ and $b\left(\cdot, 0, \delta_{0_{n}}, 0\right)\in L^{2,K}_{\mathbb{F}}(0,\infty;\mathbb{R}^{n})$.\\
(iii) There exists a constant $\kappa>K+\frac{(l_{\sigma x}+l_{\sigma \mu})^{2}}{2}+l_{b\mu}$ such that for any $t>0,\ \alpha \in A,\ \mu \in \mathcal{P}_2(\mathbb{R}^{n})$, $x, x^{\prime} \in \mathbb{R}^{n}$, it holds that
$$
\left \langle x-x^{\prime}, b(t, x, \mu, \alpha)-b\left(t, x^{\prime}, \mu, \alpha\right)\right \rangle \leq-\kappa\left|x-x^{\prime}\right|^2 .
$$  
\end{assumption}
Define $\mathcal{A}:=L^{2,K}_{\mathbb{F}}(0, \infty;A)$ to be the space of all admissible controls. For any control $\alpha \in \mathcal{A}$, it follows from Lemma \ref{x propostion} that under Assumption \ref{assumption control x}, the following controlled McKean-Vlasov SDE \eqref{general state dynamics}  admits a unique solution $X_{t}\in L^{2,K}_{\mathbb{F}}(0,\infty;\mathbb{R}^{n})$.

We consider the following mean field control problem: Minimize 
\begin{equation}\label{minimization}
J(\alpha):=\mathbb{E}\left[\int_0^{\infty} e^{2Kt} f\left(t, X_t, \mathcal{L}\left(X_t\right), \alpha_t\right) d t\right],
\end{equation}
over the set $\mathcal{A}$ of admissible control processes, which is finite under Assumption \ref{assumption control x}, subject to the dynamic constraint
\begin{equation}\label{general state dynamics}
  \left\{  \begin{aligned}
dX_t&=b\left(t, X_t, \mathcal{L}\left(X_t\right), \alpha_t\right) dt+\sigma(t,X_{t},\mathcal{L}\left(X_t\right),\alpha_{t}) d W_t, \\
X_0&=\xi.   
\end{aligned}\right.
\end{equation}
\subsection{Pontryagin's stochastic maximum principle}\label{subsection 1}
In this subsection, we aim to establish the Pontryagin's stochastic maximum principle for the infinite mean field control problem \eqref{minimization}-\eqref{general state dynamics}. When the volatility $\sigma$ of the state dynamics is a constant, the corresponding maximum principle was obtained in \cite{bayraktar2023solvability}. Now we extend the maximum principle to a more general state dynamics \eqref{general state dynamics}.

We define the Hamiltonian $\mathcal{H}$ by
\begin{equation}\label{Hamiltonian}
\mathcal{H}(t, x, \mu,y,z,\alpha):=b(t, x, \mu, \alpha)\cdot y+\sigma(t,x,\mu,\alpha)\cdot z+f(t, x, \mu, \alpha)+2Kx \cdot y,
\end{equation}
for $t\in[0,\infty),\ \alpha \in A,  x, y \in \mathbb{R}^{n},\ z \in \mathbb{R}^{n\times d},\ \mu\in\mathcal{P}_{2}(\mathbb{R}^{n})$, where the dot notation stands for the inner product in Euclidean space. Moreover, the Hamiltonian $\mathcal{H}$, which is assumed to be differentiable in $(x,\alpha,\mu)$, is said to be convex in $(x,\mu,\alpha)$ if for any $t\in[0,\infty),\  y \in \mathbb{R}^{n}, \ z \in \mathbb{R}^{n\times d}$, $(x,\alpha,\mu)$, $(x^{\prime},\alpha^{\prime},\mu^{\prime})\in \mathbb{R}^{n}\times A \times \mathcal{P}_{2}(\mathbb{R}^{n})$, we have
$$
\begin{aligned}
\mathcal{H}\left(t, x^{\prime}, \mu^{\prime}, \alpha^{\prime}, y,z\right) & \geq  \mathcal{H}(t, x, \mu, \alpha, y,z)+\partial_x \mathcal{H}(t, x, \mu,\alpha,y,z)\cdot \left(x^{\prime}-x\right) \\
& \quad+\partial_\alpha \mathcal{H}(t, x, \mu, \alpha, y,z)\cdot \left(\alpha^{\prime}-\alpha\right)+\tilde{\mathbb{E}}\left[\partial_\mu \mathcal{H}(t, x, \mu, \alpha, y,z)(\tilde{X})\cdot (\tilde{X}^{\prime}-\tilde{X})\right],
\end{aligned}
$$
where $\tilde{X}^{\prime}, \tilde{X}$ are square integrable random variables defined on $(\tilde{\Omega}, \tilde{\mathcal{F}}, \tilde{\mathbb{P}})$ and have distributions $\mu^{\prime}, \mu$ respectively.

We assume the existence of a function $(t, x, y,z,\mu) \rightarrow \hat{\alpha}(t,x,y,z,\mu) \in A$, which is Lipschitz-continuous with respect to $(x, y, z,\mu)$, uniformly in $t \in[0, \infty)$ such that:
$$
\hat{\alpha}(t,x, y,z ,\mu)=\underset{\alpha \in A}{\operatorname{argmin}} \mathcal{H}(t, x, \mu,y,z,\alpha), \quad t\in[0,\infty),\ x, y \in \mathbb{R}^{n},\ z \in \mathbb{R}^{n\times d},\ \mu\in\mathcal{P}_{2}(\mathbb{R}^{n}).
$$
The existence of such function was proven in the next subsection \ref{subsection 2} under specific assumptions on the drift $b$ and the running cost function $f$. Then, we introduce the following infinite horizon McKean-Vlasov FBSDE:
\begin{equation}\label{mfc fbsde}
\left\{\begin{aligned}
dX_t & =b\left(t, X_t,\mathcal{L}(X_{t}),\hat{\alpha}(t,X_{t},Y_{t},Z_{t},\mathcal{L}(X_{t}))\right) d t+\sigma(t,X_{t},\mathcal{L}(X_{t}),\hat{\alpha}(t,X_{t},Y_{t},Z_{t},\mathcal{L}(X_{t}))) dW_t, \\
dY_t & =-\partial_x \mathcal{H}\left(t, X_{t},\mathcal{L}(X_{t}),Y_{t},Z_{t},\hat{\alpha}(t,X_{t},Y_{t},Z_{t},\mathcal{L}(X_{t}))\right)dt+Z_t d W_t\\
&\quad -\tilde{\mathbb{E}}\left[\partial_\mu \mathcal{H}\left(t, \tilde{X}_{t},\mathcal{L}(X_{t}),\tilde{Y}_{t},\tilde{Z}_{t},\hat{\alpha}(t,\tilde{X}_{t},\tilde{Y}_{t},\tilde{Z}_{t},\mathcal{L}(X_{t}))\right)\left(X_t\right)\right] dt, \\
X_0 & =\xi,   
\end{aligned}\right. 
\end{equation}
where $(\tilde{X},\tilde{Y},\tilde{Z})$ is an independent copy of $(X,Y,Z)$ defined on the space $(\tilde{\Omega},\tilde{\mathcal{F}},\tilde{\mathbb{P}})$, which is an independent copy of $(\Omega,\mathcal{F},\mathbb{P})$ and $\tilde{\mathbb{E}}$ denotes the expectation on $(\tilde{\Omega},\tilde{\mathcal{F}},\tilde{\mathbb{P}})$. 
\begin{proposition}
\label{maximum principle}
    Let $(b, f)$ be differentiable in $(x, \mu, \alpha)$, Assumption \ref{assumption control x} holds and $\mathcal{H}$ be convex in $(x, \mu, \alpha)$, $\hat{\alpha}(\cdot,0,0,0,\delta_{0_{n}}) \in L_{\mathbb{F}}^{2,K}(0,\infty;\mathbb{R}^{n})$ and $\hat{\alpha}(t,x,y,z,\mu)$ is Lipschitz in $(x,y,z,\mu)$. Moreover, suppose infinite horizon McKean-Vlasov FBSDE \eqref{mfc fbsde} admits a unique solution $(X,Y,Z)\in L^{2,K}_{\mathbb{F}}(0,\infty;\mathbb{R}^{n+n+n \times d})$. Then we have that $J(\hat{\alpha})=\min _\alpha J(\alpha)$.
\end{proposition}

\begin{proof}
Since we assume infinite horizon McKean-Vlasov FBSDE \eqref{mfc fbsde} admits a unique solution $(X,Y,Z)$, let us denote $\theta_{t}^{\wedge} := (X_{t},Y_{t},Z_{t})$, $\tilde{\theta}_{t}^{\wedge} := (\tilde{X}_{t},\tilde{Y}_{t},\tilde{Z}_{t})$, $\Theta_t^{\wedge}:=\left(\theta_{t}^{\wedge}, \mathcal{L}\left(X_t\right), \hat{\alpha}\left(t,\theta_{t}^{\wedge},\mathcal{L}\left(X_t\right)\right)\right)$ and $\tilde{\Theta}_{t}^{\wedge} := (\tilde{\theta}^{\wedge}_{t},\mathcal{L}(X_{t}),\hat{\alpha}(t,\tilde{\theta}_{t}^{\wedge},\mathcal{L}(X_{t})))$, where $(\tilde{X}_{t},\tilde{Y}_{t},\tilde{Z}_{t})$ is an independent copy of $(X_{t},Y_{t},Z_{t})$. For an arbitrary admissible control $\alpha^{\prime}$ and its associated process $X^{\prime}$, we have that
\begin{equation}
\begin{aligned}
J(\hat{\alpha})-J\left(\alpha^{\prime}\right)& =  \mathbb{E}\left[\int_0^{\infty} e^{2K t}\left(\mathcal{H}\left(t, X_t, \mathcal{L}\left(X_t\right), \hat{\alpha}_t, Y_t,Z_{t}\right)-\mathcal{H}\left(t, X_t^{\prime}, \mathcal{L}\left(X_t^{\prime}\right), \alpha_t^{\prime}, Y_t,Z_{t}\right)\right) d t\right] \\
&\quad -\mathbb{E}\left[\int_0^{\infty} e^{2K t}\left(b\left(t, X_t, \mathcal{L}\left(X_t\right), \hat{\alpha}_t\right)-b\left(t, X_t^{\prime}, \mathcal{L}\left(X_t^{\prime}\right), \alpha_t^{\prime}\right)\right)\cdot Y_t d t\right] \\
& \quad -\mathbb{E}\left[\int_0^{\infty} e^{2K t}\left(\sigma\left(t, X_t, \mathcal{L}\left(X_t\right),\hat{\alpha}_{t}\right)-\sigma\left(t, X_t^{\prime}, \mathcal{L}\left(X_t^{\prime}\right),\alpha_{t}^{\prime}\right)\right)\cdot Z_{t} d t\right] \\
&\quad -2K \mathbb{E}\left[\int_0^{\infty} e^{2K t}\left(X_t-X_t^{\prime}\right) \cdot Y_t d t\right].
\end{aligned}\label{equation1}
\end{equation}
It can be easily seen from Lemma \ref{infinite value}, that there exists a sequence of $T_i \rightarrow \infty$ such that $$\mathbb{E}\left[e^{2K T_i}\left(X_{T_i}-X_{T_i}^{\prime}\right) \cdot Y_{T_i}\right] \rightarrow 0.
$$ Applying It\^{o}'s formula to $e^{2Kt}\left(X_{t}-X_{t}^{\prime}\right) \cdot Y_{t}$ on time interval $[0,T_{i}]$, and letting $T_i \rightarrow \infty$, we obtain that
\begin{equation}
\begin{aligned}
& \quad \mathbb{E}{\left[\int_0^{\infty} e^{2K t}\left(X_t-X_t^{\prime}\right)\cdot \left(\partial_x \mathcal{H}\left(t, \Theta_t^{\wedge}\right)+\tilde{\mathbb{E}}\left[\partial_\mu \mathcal{H}(\tilde{\Theta}_t^{\wedge})\left(X_t\right)\right]\right) d t\right] } \\
 & = \mathbb{E}\left[\int_0^{\infty} e^{2K t}\left(2K\left(X_t-X_t^{\prime}\right)+b\left(t, X_t, \mathcal{L}\left(X_t\right), \hat{\alpha}_t\right)-b\left(t, X_t^{\prime}, \mathcal{L}\left(X_t^{\prime}\right), \alpha_t^{\prime}\right)\right)\cdot Y_t d t\right]\\
 &\quad + \mathbb{E}\left[\int_0^{\infty} e^{2K t}\left(\sigma\left(t, X_t, \mathcal{L}\left(X_t\right),\hat{\alpha}_{t}\right)-\sigma\left(t, X_t^{\prime}, \mathcal{L}\left(X_t^{\prime}\right),\alpha_{t}^{\prime}\right)\right)\cdot Z_t d t\right].
\end{aligned}\label{equation2}
\end{equation}
According to the convexity of $\mathcal{H}$ and the fact that $\hat{\alpha}_t=\operatorname{argmin}_{\alpha \in A} \mathcal{H}\left(t, X_t, \mathcal{L}\left(X_t\right), \alpha,Y_t,Z_{t}\right)$, it holds that
\begin{equation}\label{equation3}
 \begin{aligned}
&\mathcal{H}\left(t, X_t^{\prime}, \mathcal{L}\left(X_t^{\prime}\right), \alpha_t^{\prime}, Y_t,Z_{t}\right)-\mathcal{H}\left(t, X_t, \mathcal{L}\left(X_t\right), \hat{\alpha}_t, Y_t,Z_{t}\right)  \\
\geq & \left(X_t^{\prime}-X_t\right)\cdot \partial_x \mathcal{H}\left(t, \Theta_t^{\wedge}\right)+\tilde{\mathbb{E}}\left[\partial_\mu \mathcal{H}\left(t, \Theta_t^{\wedge}\right)(\tilde{X}_t) \cdot (\tilde{X}_t^{\prime}-\tilde{X}_t)\right] \\
& +\left(\alpha_t^{\prime}-\hat{\alpha}_t\right)\cdot \partial_\alpha \mathcal{H}\left(t, \Theta_t^{\wedge}\right) \\
\geq & \left(X_t^{\prime}-X_t\right) \cdot \partial_x \mathcal{H}\left(t, \Theta_t^{\wedge}\right)+\tilde{\mathbb{E}}\left[\partial_\mu \mathcal{H}\left(t, \Theta_t^{\wedge}\right)(\tilde{X}_t)\cdot (\tilde{X}_t^{\prime}-\tilde{X}_t)\right] .
\end{aligned}   
\end{equation}
Using Fubini's theorem and the fact $\tilde{\Theta}_{t}^{\wedge}$ is an independent copy of $\Theta^{\wedge}_{t}$, we have that
\begin{equation}\label{fubini}
\mathbb{E}\left[\left(X_t^{\prime}-X_t\right)  \cdot \tilde{\mathbb{E}}\left[\partial_\mu \mathcal{H}(\tilde{\Theta}_t^{\wedge})\left(X_t\right)\right]\right]=\mathbb{E}\left[\ \tilde{\mathbb{E}}\left[\partial_\mu \mathcal{H}\left(t, \Theta_t^{\wedge}\right)(\tilde{X}_t) \cdot (\tilde{X}_t^{\prime}-\tilde{X}_t)\right]\right].    
\end{equation}
Combined with \eqref{equation1}, \eqref{equation2}, \eqref{equation3} and \eqref{fubini}, we conclude that
\begin{equation}
 \begin{aligned}
J(\hat{\alpha})  -J\left(\alpha^{\prime}\right) & =  \mathbb{E}\left[\int_0^{\infty} e^{2K t}\left(\mathcal{H}\left(t, X_t, \mathcal{L}\left(X_t\right), \hat{\alpha}_t, Y_t,Z_{t}\right)-\mathcal{H}\left(t, X_t^{\prime}, \mathcal{L}\left(X_t^{\prime}\right), \alpha_t^{\prime}, Y_t,Z_{t}\right)\right) d t\right] \\
& \quad -\mathbb{E}\left[\int_0^{\infty} e^{2K t}\left(X_t-X_t^{\prime}\right)\cdot \left(\partial_x \mathcal{H}\left(t, \Theta_t^{\wedge}\right)+\tilde{\mathbb{E}}\left[\partial_\mu \mathcal{H}\left(\tilde{\Theta}_t^{\wedge}\right)\left(X_t\right)\right]\right) d t\right] \\
 & =  \mathbb{E}\left[\int_0^{\infty} e^{2K t}\left(\mathcal{H}\left(t, X_t, \mathcal{L}\left(X_t\right), \hat{\alpha}_t, Y_t,Z_{t}\right)-\mathcal{H}\left(t, X_t^{\prime}, \mathcal{L}\left(X_t^{\prime}\right), \alpha_t^{\prime}, Y_t,Z_{t}\right)\right) d t\right] \\
& \quad-\mathbb{E}\left[\int_0^{\infty} e^{2K t}\left(\left(X_t-X_t^{\prime}\right)  \cdot \partial_x \mathcal{H}\left(t, \Theta_t^{\wedge}\right)+\tilde{\mathbb{E}}\left[\partial_\mu \mathcal{H}\left(t, \Theta_t^{\wedge}\right)(\tilde{X}_t)\cdot (\tilde{X}_t-\tilde{X}_t^{\prime})\right]\right) d t\right]\\
&\leq 0.
\end{aligned}   
\end{equation}
\end{proof}
\subsection{Solvability of mean field control problems}\label{subsection 2}
In this subsection, we give sufficient conditions on the given data for the existence and uniqueness of solutions to infinite horizon FBSDE \eqref{mfc fbsde}. As it is most often the case in applications of the maximum principle, we choose $A = \mathbb{R}^{m}$ and  $\mathcal{A}:=L^{2,K}_{\mathbb{F}}(0, \infty;\mathbb{R}^{m})$ to be the space of all admissible controls. We consider a linear model for the forward dynamics of the state.
\begin{assumption}\label{mfc assumption}
    (i) The drift $b$ and the volatility $\sigma$ are linear in $\mu, x$ and $\alpha$. They read
$$
\begin{aligned}
b(t, x, \mu, \alpha) & =b_0(t)+b_1(t)x +b_2(t)\bar{\mu}+b_3(t) \alpha, \\
\sigma(t, x, \mu, \alpha) & =\sigma_0(t)+\sigma_1(t) x+\sigma_2(t) \bar{\mu}+\sigma_3(t) \alpha,
\end{aligned}
$$
for some bounded measurable deterministic functions $b_0, b_1, b_2$, $b_3$ with values in $\mathbb{R}^{n},\mathbb{R}^{n\times n},\mathbb{R}^{n\times n}$ and $\mathbb{R}^{n\times m }$ and $\sigma_0, \sigma_1, \sigma_2$, $\sigma_3$ with values in $\mathbb{R}^{n\times d}$, $\mathbb{R}^{(n\times d) \times n}$, $\mathbb{R}^{(n\times d) \times n}$ and $\mathbb{R}^{(n\times d) \times m}$(the parentheses around $n \times d$ indicating that $\sigma_i(t) u_i$ is seen as an element of $\mathbb{R}^{n \times d}$ whenever $u_i \in \mathbb{R}^n$, with $i=1,2$, or $u_i \in \mathbb{R}^{m}$, with $i=3$), $b_{0}(\cdot) \in L^{2,K}_{\mathbb{F}}(0,\infty;\mathbb{R}^{n})$, $\sigma_{0}(\cdot)\in  L^{2,K}_{\mathbb{F}}(0,\infty;\mathbb{R}^{n\times d}) $ and  we use the notation $\bar{\mu}=\int x d \mu(x)$ for the mean of a measure $\mu$.\\
(ii) $f$ is differentiable with respect to $(x,\mu, \alpha)$ and $|f(\cdot,0,\delta_{0_{n}},0)|\in L^{2,K}_{\mathbb{F}}(0,\infty;\mathbb{R})$. Moreover, the derivatives satisfy that $\partial_{x}f(\cdot,0,\delta_{0_{n}},0), \partial_{\mu}f(\cdot,0,0,\delta_{0_{n}})(0) \in L^{2,K}_{\mathbb{F}}(0,\infty;\mathbb{R}^{n})$ and $\partial_{\alpha}f(\cdot,0,\delta_{0_{n}},0)\in L^{2,K}_{\mathbb{F}}(0,\infty;\mathbb{R}^{m})$ .\\
(iii) There exists a positive constant $\tilde{L}$ such that for all $t\geq 0$, the functions $\partial_x f, \partial_\alpha f$ are $\tilde{L}$-Lipschitz continuous with respect to $(x,\alpha,\mu)$. Moreover, there exists a version of $\partial_\mu f(t, x^{\prime}, \mu,\alpha)(\cdot)$ such that $\partial_{\mu}f(t,x^{\prime},\mu,\alpha)(x)$ is $\tilde{L}$-Lipschitz in $(x^{\prime},\mu,\alpha,x)$ \footnote{ For the Lipschitz property of
$\left(t,x^{\prime}, \mu, \alpha, x\right) \mapsto \partial_\mu f\left(t, x^{\prime}, \mu, \alpha\right)(x)$, \cite[Lemma 5.41]{meanfield2018book1} provides a simple criterion.}, the Lipschitz property in the variable $\mu$ being understood in the sense of the $2$-Wassertein distance.\\
(iv) The function $f$ is convex with respect to $(x, \mu, \alpha)$ for $t \geq 0 $ in such a way that, for some $\lambda>0$,
$$
\begin{aligned}
& f\left(t, x^{\prime}, \mu^{\prime}, \alpha^{\prime}\right)-f(t, x, \mu, \alpha)-\partial_{(x, \alpha)} f(t, x, \mu, \alpha) \cdot\left(x^{\prime}-x, \alpha^{\prime}-\alpha\right)\\
&\quad -\tilde{\mathbb{E}}\left[\partial_\mu f(t, x, \mu, \alpha)(\tilde{X}) \cdot(\tilde{X}^{\prime}-\tilde{X})\right]\geq \lambda\left|\alpha^{\prime}-\alpha\right|^2,
\end{aligned}
$$
whenever $\tilde{X}, \tilde{X}^{\prime}$ with distributions $\mu$ and $\mu^{\prime}$, respectively. 
\end{assumption}
Now, we introduce some notations about bounded measurable functions valued in $\mathbb{R}^{n\times n}$. It is obvious that for any $t>0$, the matrix $(b_{1}(t)+b_{1}(t)^{\top})$ is symmetrical. Let $\lambda_{\max} (b_{1}(t)+b_{1}(t)^{\top})$ denote the largest eigenvalue of $(b_{1}(t)+b_{1}(t)^{\top})$. Due to the boundedness of $b_{1}(t)$, it is clear that
\begin{equation}
   \frac{1}{2}\sup_{t\in[0,\infty)}\lambda_{\max} (b_{1}(t)+b_{1}(t)^{\top}) <+\infty.
\end{equation}
Then, we have
\begin{equation}
\begin{aligned}
    \left \langle b_{1}(t)x,x \right \rangle   & = \frac{1}{2}\left\langle (b_{1}(t)+b_{1}(t)^{\top})x,x \right\rangle \leq \frac{1}{2}\lambda_{\max}(b_{1}(t)+b_{1}(t)^{\top})|x|^{2}\\ 
    & \leq \frac{1}{2}\sup_{t\in[0,\infty)}\lambda_{\max} (b_{1}(t)+b_{1}(t)^{\top})|x|^{2},
\end{aligned}
\end{equation}
for any $t>0$ and $x\in \mathbb{R}^{n}$.\\
Then in this linear setting, to ensure the control problem well-defined,  we can choose $\kappa$ in  Assumption \ref{assumption control x} (iii) as
\begin{equation}
    \kappa  = - \frac{1}{2}\sup_{t\in[0,\infty)}\lambda_{\max} (b_{1}(t)+b_{1}(t)^{\top}),
\end{equation}
and the corresponding relationship of $\kappa$ and ${K}$ is reduced to 
\begin{equation}\label{stateb1b2 condition}
   K<-\frac{1}{2}\sup_{t\in[0,\infty)}\lambda_{\max} (b_{1}(t)+b_{1}(t)^{\top}) -\|b_{2}(\cdot)\|_{\infty}-\frac{(\|\sigma_{1}(\cdot)\|_{\infty}+\|\sigma_{2}(\cdot)\|_{\infty})^{2}}{2}.
\end{equation}
Since the drift and the volatility are linear, the Hamiltonian defined as \eqref{Hamiltonian} takes the following particular form:
\begin{equation}\label{specialhamiltonian}
\begin{aligned}
H(t, x, \mu, y, z, \alpha)= & \left[b_0(t)+b_1(t) x+b_2(t) \bar{\mu}+b_3(t) \alpha\right]\cdot y \\
& +\left[\sigma_0(t)+\sigma_1(t)x+\sigma_2(t) \bar{\mu}+\sigma_3(t) \alpha\right] \cdot z+f(t, x, \mu, \alpha)+2Kx\cdot y, 
\end{aligned}    
\end{equation}
for $t\in [0,\infty),\ \alpha\in \mathbb{R}^{m},\ x,y\in \mathbb{R}^{n},\ z \in \mathbb{R}^{n\times d},\ \mu \in \mathcal{P}_{2}(\mathbb{R}^{n})$.

First, using similar argument in \cite[Lemma 3.3, Lemma 6.18]{meanfield2018book1}, we can easily obtain the following result about the minimization of the Hamiltonian  \eqref{specialhamiltonian}.
\begin{lemma}\label{alphaLipschitz}
Let Assumption \ref{mfc assumption} holds. Then for any $(t,x,\mu,y,z,\alpha) \in [0,\infty)\times \mathbb{R}^{n}\times \mathcal{P}_{2}(\mathbb{R}^{n})\times \mathbb{R}^{n} \times \mathbb{R}^{n\times d}\times \mathbb{R}^{m}$, there exists a unique minimizer $\hat{\alpha}(t,x,\mu,y,z)$ of Hamiltonian $H$. Moreover, the function $[0, \infty) \times \mathbb{R}^{n} \times \mathcal{P}_2\left(\mathbb{R}^{n}\right) \times \mathbb{R}^{n} \times \mathbb{R}^{n\times d }\ni(t, x, \mu, y, z) \mapsto$ $\hat{\alpha}(t,x, \mu, y, z) \in \mathbb{R}^{m}$ is measurable, locally bounded and Lipschitz continuous with respect to $(x, \mu, y, z)$, uniformly in $t \in[0, \infty)$, the Lipschitz constant depending only upon $\lambda$, $\|b_{3}(\cdot)\|_{\infty}$, $\|\sigma_{3}(\cdot)\|_{\infty}$ and the Lipschitz constant $\tilde{L}$ of $\partial_\alpha f$ in $(x, \mu)$. In fact, an explicit upper bound for $\hat{\alpha}$ reads:
\begin{equation}
    \begin{aligned}
        \forall (t,x,\mu,y,z) & \in [0,\infty)\times \mathbb{R}^{n} \times \mathcal{P}_{2}(\mathbb{R}^{n}) \times \mathbb{R}^{n} \times \mathbb{R}^{n\times d}, \\
        &|\hat{\alpha}(t,x,\mu,y,z)| \leq \lambda^{-1}(|\partial_{\alpha}f(t,x,\mu,\beta)|+|b_{3}(t)|y|+|\sigma_{3}(t)||z|)+|\beta_{t}|,
    \end{aligned}
\end{equation}
where $\beta_{t}$ is any admissible control in $\mathcal{A}$, and then $\hat{\alpha}(\cdot,0,\delta_{0_{n}},0,0)\in L^{2,K}_{\mathbb{F}}(0,\infty;\mathbb{R}^{m})$.  
\end{lemma}
Then the infinite horizon McKean-Vlasov FBSDE corresponding to \eqref{mfc fbsde} reads
\begin{equation}\label{linearmfcfbsde}
\left\{
\begin{aligned}
d X_t &= {\left[b_0(t)+b_1(t) X_t+b_2(t) \mathbb{E}\left[X_t\right]+b_3(t) \hat{\alpha}\left(t, X_t, \mathcal{L}(X_{t}), Y_t, Z_t\right)\right] d t } \\
& \quad +\left[\sigma_0(t)+\sigma_1(t) X_t+\sigma_2(t) \mathbb{E}\left[X_t\right]+\sigma_3(t) \hat{\alpha}\left(t, X_t, \mathcal{L}(X_{t}), Y_t, Z_t\right)\right] d W_t,\\
d Y_t & =  -\left[\partial_x f\left(t, X_t, \mathcal{L}(X_{t}), \hat{\alpha}\left(t, X_t, \mathcal{L}(X_{t}), Y_t, Z_t\right)\right)+b_1(t)Y_t+2KY_{t}+\sigma_1(t) Z_t\right] d t+Z_t d W_t \\
&\quad -\left\{\tilde{\mathbb{E}}\left[\partial_\mu f(t, \tilde{X}_t, \mathcal{L}(X_{t}), \hat{\alpha}(t, \tilde{X}_t, \mathcal{L}(X_{t}), \tilde{Y}_t, \tilde{Z}_t))\left(X_t\right)\right]+b_2(t) \mathbb{E}\left[Y_t\right]+\sigma_2(t) \mathbb{E}\left[Z_t\right]\right\} d t. 
\end{aligned}\right.
\end{equation}

Next, based on Lemma \ref{alphaLipschitz}, we can follow the arguments in the proof of \cite[Lemma 3.2]{bayraktar2023solvability} to prove that the following function is Lipschitz: 
\begin{equation}
\begin{aligned}
    \Psi(t, x, m) &:= \tilde{\mathbb{E}}\left[\partial_\mu f(t, \tilde{X}_t, \mathcal{L}(X_{t}), \hat{\alpha}(t, \tilde{X}_t, \mathcal{L}(X_{t}), \tilde{Y}_t, \tilde{Z}_t))\left(X_t\right)\right]\\
   & = \int_{x^{\prime}, y^{\prime},z^{\prime}} \partial_\mu f\left(t, x^{\prime}, \mu, \hat{\alpha}\left(t,x^{\prime},\mu, y^{\prime}, z^{\prime}\right)\right)(x) d m\left(x^{\prime}, y^{\prime},z^{\prime}\right),
\end{aligned}
\end{equation}
where $m\in \mathcal{P}_{2}(\mathbb{R}^{n+n+n\times d})$ and $\mu$ is the first marginal of $m$. To avoid repetition, the detailed proof of the following lemma is omitted.
 \begin{lemma}\label{measureLipschitz}
 Under Assumption \ref{mfc assumption}, for any $t>0$, $x, \bar{x} \in \mathbb{R}^{n}, m, \bar{m} \in \mathcal{P}_2\left(\mathbb{R}^{n+n+n \times d}\right)$, it holds that
\begin{equation}
    |\Psi(t, x, m)-\Psi(t, \bar{x}, \bar{m})| \leq C_{\Psi}\mathcal{W}_2(m, \bar{m})+\tilde{L}|x-\bar{x}|,
\end{equation}
where $C_{\Psi}$ depending only upon the Lipschitz constant of $\hat{\alpha}$ in $(x,\mu,y,z)$, the Lipschitz constant $\tilde{L}$ of $\partial_{\mu}f$ in $(x^{\prime},\mu,\alpha,x)$.
\end{lemma}
Now we give the main result of this section. It is worth noting that the condition \eqref{b1b2condition} for parameter $K$ is exactly the requirement \eqref{stateb1b2 condition} to ensure the infinite horizon mean field control problems well-defined. 
\begin{Theorem}\label{theorem mfc solvability}
   Suppose Assumption \ref{mfc assumption} holds. Let
\begin{equation}\label{b1b2condition}
   K<-\frac{1}{2}\sup_{t\in[0,\infty)}\lambda_{\max} (b_{1}(t)+b_{1}(t)^{\top}) -\|b_{2}(\cdot)\|_{\infty}-\frac{(\|\sigma_{1}(\cdot)\|_{\infty}+\|\sigma_{2}(\cdot)\|_{\infty})^{2}}{2}.
\end{equation}
 Then infinite horizon FBSDE \eqref{linearmfcfbsde} admits a unique solution $(X,Y,Z)\in L^{2,K}_{\mathbb{F}}(0,\infty;\mathbb{R}^{n+n+n\times d})$ and $\hat{\alpha}(t,X_{t},\mathcal{L}(X_{t}),Y_{t},Z_{t}),\ t\in[0,\infty)$ is the optimal control of the infinite horizon mean field control problem \eqref{minimization}-\eqref{general state dynamics}.
\end{Theorem}
\begin{proof}
Under Assumption \ref{mfc assumption}, by Lemma \ref{alphaLipschitz}, $\hat{\alpha}_{t}$ is Lipschitz in $(x,\mu,y,z)$ and $\hat{\alpha}(\cdot,0,0,0,\delta_{0_{n}}) \in L^{2,K}_{\mathbb{F}}(0,\infty;\mathbb{R}^{m})$. 
  The linearity of $(b, \sigma)$ and the  convexity of $f$ in Assumption \ref{mfc assumption} imply that the Hamiltonian $H$ is convex in $(x,\mu,\alpha)$. Indeed, for all $(t, y, z) \in[0, T] \in \mathbb{R}^{n}\times \mathbb{R}^{n \times d}$, $(x, \mu, \alpha),\ \left(x^{\prime}, \mu^{\prime}, \alpha^{\prime}\right) \in \mathbb{R}^{n} \times \mathcal{P}_2\left(\mathbb{R}^{n}\right) \times \mathbb{R}^{m}$, we have
\begin{equation}\label{hamiltonian convex}
\begin{aligned}
&\quad  H\left(t, x^{\prime}, \alpha^{\prime}, \mu^{\prime}, y, z\right)-H(t, x, a, \mu
, y, z)-\left\langle\partial_{(x, \alpha)} H(t, x, \alpha, \mu, y, z),\left(x^{\prime}-x, \alpha^{\prime}-\alpha\right)\right\rangle \\
& \quad-\tilde{\mathbb{E}}\left[\langle\partial_\mu H(t, x, \alpha, \mu, y, z)(\tilde{X}), \tilde{X}^{\prime}-\tilde{X}\rangle\right] \\
& =  b_{2}(\bar{\mu}^{\prime}-\bar{\mu})\cdot y+\sigma_{2}(\bar{\mu}^{\prime}-\bar{\mu})\cdot z+f(t,x^{\prime},\alpha^{\prime},\mu^{\prime})-f(t,x,\alpha,\mu)\\
&\quad -\tilde{\mathbb{E}}\left[\langle b_{2}y+\sigma_{2}z+\partial_{\mu}f(t,x,\alpha,\mu)(\tilde{X}),\tilde{X}^{\prime}-\tilde{X}\rangle\right]\\
& =f(t,x^{\prime},\alpha^{\prime},\mu^{\prime})- f(t,x,\alpha,\mu)-\tilde{\mathbb{E}}\left[\langle \partial_{\mu}f(t,x,\alpha,\mu)(\tilde{X}),\tilde{X}^{\prime}-\tilde{X}\rangle\right]\\
& \geq \lambda\left|\alpha^{\prime}-\alpha\right|^2,
\end{aligned}    
\end{equation}
whenever $\tilde{X},\tilde{X}^{\prime} \in L^2\left(\tilde{\Omega}, \tilde{\mathcal{F}}, \tilde{\mathbb{P}} ; \mathbb{R}\right)$ with distributions $\mu$ and $\mu^{\prime}$ and we use the notation $\bar{\mu
},\bar{\mu
}^{\prime}$ for the mean of a measure $\mu$, $\mu^{\prime}$, respectively.

Then, from Proposition \ref{maximum principle}, it remains to prove the well-posedness of FBSDE \eqref{linearmfcfbsde} and we will apply Theorem \ref{global theorem} to obtain the solvability results. By Assumption \ref{mfc assumption}, Lemma \ref{alphaLipschitz} and Lemma \ref{measureLipschitz}, it is easy to show Assumption (H1) holds. Now it remains to verify Assumption (H2) and condition \eqref{kxkyk} to obtain the desired conclusion. \\
For simplicity of notation, we denote $\Theta_{1} := (X_{1},Y_{1},Z_{1}),\ \Theta_{2}:=(X_{2},Y_{2},Z_{2})\in L^{2}(\Omega;\mathbb{R}^{n+n+n\times d})$ and 
\begin{equation}
    \alpha_{i} = \hat{\alpha}(t,X_{i},\mathcal{L}(X_{i}),Y_{i},Z_{i}),\quad  \tilde{\alpha}_{i} = \hat{\alpha}(t,\tilde{X}_{i},\mathcal{L}(X_{i}),\tilde{Y}_{i}, \tilde{Z}_{i}),\quad i = 1,2.
\end{equation}

First, we show the monotonicity condition \eqref{monotonicity} holds. It is clear that the  Hamiltonian system \eqref{linearmfcfbsde} is a special case of FBSDE \eqref{infinite FBSDE} with
\begin{equation}
 \begin{aligned}
     B(t,\Theta,\mathcal{L}(\Theta)) &= b_0(t)+b_1(t) X+b_2(t) \mathbb{E}\left[X\right]+b_3(t) \hat{\alpha}\left(t, X, \mathcal{L}(X), Y, Z\right) \\
     \sigma(t,\Theta,\mathcal{L}(\Theta)) &= \sigma_0(t)+\sigma_1(t) X+\sigma_2(t) \mathbb{E}\left[X\right]+\sigma_3(t) \hat{\alpha}\left(t, X, \mathcal{L}(X), Y, Z\right)\\
F(t,\Theta,\mathcal{L}(\Theta)) &= -\Big\{\partial_x f\left(t, X, \mathcal{L}(X), \hat{\alpha}\left(t, X, \mathcal{L}(X), Y, Z\right)\right)+b_1(t)Y+2KY+\sigma_1(t) Z\\
&\quad +\tilde{\mathbb{E}}\left[\partial_\mu f(t, \tilde{X}, \mathcal{L}(X), \hat{\alpha}(t, \tilde{X}, \mathcal{L}(X), \tilde{Y}, \tilde{Z}))\left(X\right)\right]+b_2(t) \mathbb{E}\left[Y\right]+\sigma_2(t) \mathbb{E}\left[Z\right]\Big\}.
 \end{aligned}   
\end{equation}
Let $\kappa_{x},\kappa_{y}$ be determined later, $K = \frac{\kappa_{x}+\kappa_{y}}{2}$, $G = \mathbb{I}_{n}$ and choose the measurable function $\phi_{2}$ appearing in Assumption (H2) to be
$$
\phi_{2}(t,\Theta_{1},\Theta_{2},\mathcal{L}(\Theta_{1}),\mathcal{L}(\Theta_{2})) := \mathbb{E}\left[\left|\hat{\alpha}(t,X_{1},\mathcal{L}(X_{1}),Y_{1},Z_{1})-\hat{\alpha}(t,X_{2},\mathcal{L}(X_{2}),Y_{2},Z_{2})\right|^{2}\right].
$$
Denote
\begin{equation}\label{tildeH}
\begin{aligned}
    \bar{H}(t, x, \mu, y, z, \alpha) &= \left[b_0(t)+b_1(t) x+b_2(t) \bar{\mu}+b_3(t) \alpha\right]\cdot y +\left[\sigma_0(t)+\sigma_1(t)x+\sigma_2(t) \bar{\mu}+\sigma_3(t) \alpha\right]\cdot z\\
    &\quad +f(t, x, \mu,\alpha).
\end{aligned}
 \end{equation}
 It is obvious that $ H(t,x,\mu,y,z,\alpha) = \bar{H}(t,x,\mu,y,z,\alpha)+2Kx\cdot y$ and it has the same convex property as \eqref{hamiltonian convex}.
 From the definition of $B$ and the linearity of $\bar{H}$ in $(y, z)$, we can deduce that
\begin{equation}
\begin{aligned}
& \quad \mathbb{E}\Big[ \left\langle B\left(t, \Theta_1, \mathcal{L}(\Theta_{1})\right)-B\left(t, \Theta_2, \mathcal{L}(\Theta_{2})\right), Y_1-Y_2\right\rangle + \left \langle \sigma\left(t, \Theta_1, \mathcal{L}(\Theta_{1})\right)-\sigma\left(t, \Theta_2, \mathcal{L}(\Theta_{2})\right), Z_1-Z_2 \right\rangle \Big]\\
& =  \mathbb{E}\Big[\bar{H}\left(t, X_1, \alpha_1, \mathcal{L}(X_{1}), Y_1, Z_1\right)-\bar{H}\left(t, X_1, \alpha_1, \mathcal{L}(X_{1}), Y_2, Z_2\right)\Big] \\
& \quad-\mathbb{E}\Big[\bar{H}\left(t, X_2, \alpha_2, \mathcal{L}(X_{2}), Y_1, Z_1\right)-\bar{H}\left(t, X_2, \alpha_2, \mathcal{L}(X_{2}), Y_2, Z_2\right)\Big].
\end{aligned}    
\end{equation}
Moreover, by the definition of $F$, we can obtain that
\begin{equation}
 \begin{aligned}
& \quad \mathbb{E}  {\left[\left\langle F\left(t, \Theta_1, \mathcal{L}(\Theta_{1})\right)-F\left(t, \Theta_2, \mathcal{L}(\Theta_{2})\right), X_1-X_2\right\rangle\right] } \\
& =  -\mathbb{E}\left[\left\langle\partial_x \bar{H}\left(t, X_1, \alpha_1, \mathcal{L}(X_{1}), Y_1, Z_1\right), X_1-X_2\right\rangle\right] \\
& \quad-\mathbb{E}\Big[\tilde{\mathbb{E}}[\langle\partial_\mu \bar{H}\left(t, X_1, \alpha_1, \mathcal{L}(X_{1}), Y_1, Z_1\right)(\tilde{X}_1), \tilde{X}_1-\tilde{X}_2\rangle]\Big] \\
& \quad+\mathbb{E}\left[\langle \partial_x \bar{H}\left(t, X_2, \alpha_2, \mathcal{L}(X_{2}), Y_2, Z_2\right), X_1-X_2\rangle\right] \\
& \quad+\mathbb{E}\Big[\tilde{\mathbb{E}}[\langle\partial_\mu \bar{H}\left(t, X_2, \alpha_2, \mathcal{L}(X_{2}), Y_2, Z_2\right)(\tilde{X}_2), \tilde{X}_1-\tilde{X}_2\rangle]\Big]\\
& \quad-2K\mathbb{E}\left[\left \langle X_{1}-X_{2}, Y_{1}-Y_{2} \right \rangle \right],
\end{aligned}   
\end{equation}
where we have also applied Fubini's theorem and the fact that $\mathcal{L}\left(X_i, Y_i, Z_i, \alpha_i\right)=\mathcal{L}(\tilde{X}_i, \tilde{Y}_i, \tilde{Z}_i, \tilde{\alpha}_i)$ for $i=1,2$.
Therefore, we can conclude that
\begin{equation*}
 \begin{aligned}
&  \quad \mathbb{E} \left[\left\langle  B\left(t, \Theta_1, \mathcal{L}(\Theta_{1})\right)-B\left(t, \Theta_2, \mathcal{L}(\Theta_{2})\right), Y_1-Y_2\right\rangle\right]\\
&\quad \quad+\mathbb{E}\left[\left\langle\sigma\left(t, \Theta_1, \mathcal{L}(\Theta_{1})\right)-\sigma\left(t, \Theta_2, \mathcal{L}(\Theta_{2})\right), Z_1-Z_2\right\rangle\right] \\
&\quad \quad +\mathbb{E}\left[\left\langle F\left(t, \Theta_1, \mathcal{L}(\Theta_{1})\right)-F\left(t, \Theta_2, \mathcal{L}(\Theta_{2})\right), X_1-X_2\right\rangle\right]\\
&\quad \quad+2K\mathbb{E}[\left \langle X_{1}-X_{2}, Y_{1}-Y_{2}\right \rangle ] \\
& =  \mathbb{E}\Bigg[\bar{H}\left(t, X_1, \alpha_1, \mathcal{L}(X_{1}), Y_1, Z_1\right)-\bar{H}\left(t, X_2, \alpha_2, \mathcal{L}(X_{2}), Y_1, Z_1\right) \\
& \quad -\left\langle\partial_x \bar{H}\left(t, X_1, \alpha_1, \mathcal{L}(X_{1}), Y_1, Z_1\right), X_1-X_2\right\rangle \\
&\quad -\tilde{\mathbb{E}}\left[\langle\partial_\mu \bar{H}\left(t, X_1, \alpha_1, \mathcal{L}(X_{1}), Y_1, Z_1\right)(\tilde{X}_1), \tilde{X}_1-\tilde{X}_2\rangle\right]\Bigg] \\
& \quad-\mathbb{E}\Bigg[\bar{H}\left(t, X_1, \alpha_1, \mathcal{L}(X_{1}), Y_2, Z_2\right)-\bar{H}\left(t, X_2, \alpha_2, \mathcal{L}(X_{2}), Y_2, Z_2\right) \\
& \quad-\left\langle\partial_x \bar{H}\left(t, X_2, \alpha_2, \mathcal{L}(X_{2}), Y_2, Z_2\right), X_1-X_2\right\rangle \\
& \quad -\tilde{\mathbb{E}}\left[\langle\partial_\mu \bar{H}\left(t, X_2, \alpha_2, \mathcal{L}(X_{2}), Y_2, Z_2\right)(\tilde{X}_2), \tilde{X}_1-\tilde{X}_2\rangle\right]\Bigg] \\
& \leq -2\lambda \mathbb{E}\left[\left|\alpha_1-\alpha_2\right|^2\right]\\
&\leq -2\lambda \phi_{2}(t,\Theta_{1},\Theta_{2},\mathcal{L}(\Theta_{1}),\mathcal{L}(\Theta_{2})),
\end{aligned}   
\end{equation*}
 which corresponds to $\beta_{2}>0$ in monotonicity condition \eqref{monotonicity}. Then, we check the condition of case 1 in Assumption (H2)(ii). \\
For the coefficient function $\sigma$, for any $\varepsilon_{\sigma}>0$, we have
\begin{equation}\label{sigma}
\begin{aligned}
         &\quad \mathbb{E}\left[ |\sigma\left(t, \Theta_1, \mathcal{L}(\Theta_{1})\right)-\sigma(t,\Theta_{2},\mathcal{L}(\Theta_{2}))|^{2}\right]\\
         & = \mathbb{E}\left[\Big|\sigma_{1}(t)\left(X_{1}-X_{2}\right)+\sigma_{2}(t)\left(\mathbb{E}[X_{1}]-\mathbb{E}[X_{2}]\right)+\sigma_{3}(t)(\alpha_{1}-\alpha_{2})\Big|^{2}\right]\\
         & \leq \left(\left(\|\sigma_{1}(\cdot)\|_{\infty}+\|\sigma_{2}(\cdot)\|_{\infty}\right)^{2}+\varepsilon_{\sigma}\right)\mathbb{E}[|X_{1}-X_{2}|^{2}]\\
         &\quad +\left(\|\sigma_{3}(\cdot)\|_{\infty}^{2}+\frac{1}{4\varepsilon_{\sigma}}\|\sigma_{3}(\cdot)\|_{\infty}^{2}(\|\sigma_{1}(\cdot)\|_{\infty}^{2}+\|\sigma_{2}(\cdot)\|_{\infty}^{2})\right)\phi_{2}(t,\Theta_{1},\Theta_{2},\mathcal{L}(\Theta_{1}),\mathcal{L}(\Theta_{2})),\\
         & =  l_{\sigma}\mathbb{E}[|X_{1}-X_{2}|^{2}]+\left(\|\sigma_{3}(\cdot)\|_{\infty}^{2}+\frac{1}{4\varepsilon_{\sigma}}\|\sigma_{3}(\cdot)\|_{\infty}^{2}(\|\sigma_{1}(\cdot)\|_{\infty}^{2}+\|\sigma_{2}(\cdot)\|_{\infty}^{2})\right)\phi_{2}(t,\Theta_{1},\Theta_{2},\mathcal{L}(\Theta_{1}),\mathcal{L}(\Theta_{2})),
    \end{aligned}
\end{equation}
where
$$
l_{\sigma} =  \left(\|\sigma_{1}(\cdot)\|_{\infty}+\|\sigma_{2}(\cdot)\|_{\infty}\right)^{2}+\varepsilon_{\sigma}.
$$
As for coefficient function $B$, we have
\begin{equation}\label{xkppax}
\begin{aligned}
    &\quad  \mathbb{E}\left[\left\langle B\left(t, \Theta_1, \mathcal{L}(\Theta_{1})\right)-B\left(t, \Theta_2, \mathcal{L}(\Theta_{2})\right), X_1-X_2\right\rangle\right] \\
    &= \mathbb{E}[\langle b_{1}(t)(X_{1}-X_{2})+b_{2}(t)(\mathbb{E}[X_{1}]-\mathbb{E}[X_{2}])+b_{3}(t)(\alpha_{1}-\alpha_{2}), X_{1}-X_{2} \rangle] \\
    &\leq\left(\frac{1}{2}\sup_{t\in[0,\infty)}\lambda_{\max} (b_{1}(t)+b_{1}(t)^{\top})+\|b_{2}(\cdot)\|_{\infty}+\varepsilon_{1}\right)\mathbb{E}[|X_{1}-X_{2}|^{2}]\\
    &\quad \quad+\frac{\|b_{3}(\cdot)\|_{\infty}}{4\varepsilon_{1}}\phi_{2}(t,\Theta_{1},\Theta_{2},\mathcal{L}(\Theta_{1}),\mathcal{L}(\Theta_{2}))\\
    &\leq -\kappa_{x}\mathbb{E}[|X_{1}-X_{2}|^{2}]+\frac{\|b_{3}(\cdot)\|_{\infty}}{4\varepsilon_{1}}\phi_{2}(t,\Theta_{1},\Theta_{2},\mathcal{L}(\Theta_{1}),\mathcal{L}(\Theta_{2})),
\end{aligned}
\end{equation}
where the $\kappa_{x}$ satisfying
\begin{equation}\label{kappax}
     \kappa_{x}<-\frac{1}{2}\sup_{t\in[0,\infty)}\lambda_{\max} (b_{1}(t)+b_{1}(t)^{\top})-\|b_{2}(\cdot)\|_{\infty},
\end{equation}
and the condition \eqref{kappax} ensures the existence of $\varepsilon_{1}>0 $ such that last inequality of \eqref{xkppax} holds.\\ 
Moreover, by the Lipschitz property of $\partial_{x}f$ and $\partial_{\mu}f$, for any $\varepsilon_{z}>0$, we have
     \begin{align}\label{ykppay}
        & \quad \mathbb{E}\left[\langle F(t,\Theta_{1},\mathcal{L}(\Theta_{1}))-F(t,\Theta_{2},\mathcal{L}(\Theta_{2})), Y_1-Y_2\rangle\right] \notag \\ 
        & =  \mathbb{E}\Bigg[\Big\langle \partial_x f\left(t, X_2, \mathcal{L}(X_{2}),\alpha_{2}\right)-\partial_x f\left(t, X_1, \mathcal{L}(X_{1}),\alpha_{1}\right)+b_{1}(t) (Y_{2}-Y_{1})+b_2(t) (\mathbb{E}\left[Y_2\right]-\mathbb{E}\left[Y_1\right]) \notag\\
        & \quad +2K(Y_{2}-Y_{1})+\tilde{\mathbb{E}}\big[\partial_\mu f(t, \tilde{X}_2, \mathcal{L}(X_{2}), \tilde{\alpha}_{2})\left(X_2\right)\big]-\tilde{\mathbb{E}}[\partial_\mu f(t, \tilde{X}_1, \mathcal{L}(X_{1}), \tilde{\alpha}_{1})\left(X_1\right)] \notag\\
        &\quad +\sigma_{1}(t)(Z_{1}-Z_{2})+\sigma_{2}(t)(\mathbb{E}[Z_{1}]-\mathbb{E}[Z_{2}]),Y_{1}-Y_{2}\Big\rangle \Bigg]  \notag\\
       &  \geq  -\left(\frac{1}{2}\sup_{t\in[0,\infty)}\lambda_{\max} (b_{1}(t)+b_{1}(t)^{\top})+\|b_{2}(\cdot)\|_{\infty}+2K+7\varepsilon_{2}\right)\mathbb{E}[|Y_{1}-Y_{2}|^{2}] \notag\\
       &\quad -\frac{\left(\|\sigma_{1}(\cdot)\|_{\infty}+\|\sigma_{2}(\cdot)\|_{\infty}\right)^{2}+\varepsilon_{z}}{2}\mathbb{E}[|Y_{1}-Y_{2}|^{2}] -\frac{\left(\|\sigma_{1}(\cdot)\|_{\infty}+\|\sigma_{2}(\cdot)\|_{\infty}\right)^{2}}{2\left(\left(\|\sigma_{1}(\cdot)\|_{\infty}+\|\sigma_{2}(\cdot)\|_{\infty}\right)^{2}+\varepsilon_{z}\right)} \mathbb{E}\left[|Z_{1}-Z_{2}|^{2}\right] \notag\\
        & \quad-\frac{5\tilde{L}^{2}}{4\varepsilon_{2}}\mathbb{E}[|X_{1}-X_{2}|^{2}]-\frac{2\tilde{L}^{2}}{4\varepsilon_{2}}\phi_{2}(t,\Theta_{1},\Theta_{2},\mathcal{L}(\Theta_{1}),\mathcal{L}(\Theta_{2})) \notag\\
       & \geq -\left(\kappa_{y}+\frac{l_{z}}{2}\right)\mathbb{E}[|Y_{1}-Y_{2}|^{2}]-\frac{\gamma}{2} \mathbb{E}\left[|Z_{1}-Z_{2}|^{2}\right] -\frac{5\tilde{L}^{2}}{4\varepsilon_{2}}\mathbb{E}[|X_{1}-X_{2}|^{2}]\notag\\
       &\quad \quad -\frac{2\tilde{L}^{2}}{4\varepsilon_{2}}\phi_{2}(t,\Theta_{1},\Theta_{2},\mathcal{L}(\Theta_{1}),\mathcal{L}(\Theta_{2})),
        \end{align}   
where 
$$ l_{z} = \left(\|\sigma_{1}(\cdot)\|_{\infty}+\|\sigma_{2}(\cdot)\|_{\infty}\right)^{2}+\varepsilon_{z},\quad \gamma = \frac{\left(\|\sigma_{1}(\cdot)\|_{\infty}+\|\sigma_{2}(\cdot)\|_{\infty}\right)^{2}}{\left(\left(\|\sigma_{1}(\cdot)\|_{\infty}+\|\sigma_{2}(\cdot)\|_{\infty}\right)^{2}+\varepsilon_{z}\right)}
$$ 
satisfying $0<\gamma<1$ and $\kappa_{y}$ satisfying
\begin{equation}\label{kappay}
    \kappa_{y}>\frac{1}{2}\sup_{t\in[0,\infty)}\lambda_{\max} (b_{1}(t)+b_{1}(t)^{\top})+\|b_{2}(\cdot)\|_{\infty}+2K,
\end{equation}
and the condition \eqref{kappay} ensures the existence of $\varepsilon_{2}>0$ such that last inequality of \eqref{ykppay} holds.\\
In fact, for any $K$ satisfying condition \eqref{b1b2condition}, we can find $\delta>0$ such that
$$
K=-\frac{1}{2}\sup_{t\in[0,\infty)}\lambda_{\max} (b_{1}(t)+b_{1}(t)^{\top}) -\|b_{2}(\cdot)\|_{\infty}-\frac{(\|\sigma_{1}(\cdot)\|_{\infty}+\|\sigma_{2}(\cdot)\|_{\infty})^{2}}{2}-\delta.
$$
Let
$$
\kappa_{x} = -\frac{1}{2}\sup_{t\in[0,\infty)}\lambda_{\max} (b_{1}(t)+b_{1}(t)^{\top})-\|b_{2}(\cdot)\|_{\infty} -\frac{\delta}{2},
$$
and 
$$
\kappa_{y} = \frac{1}{2}\sup_{t\in[0,\infty)}\lambda_{\max} (b_{1}(t)+b_{1}(t)^{\top})+\|b_{2}(\cdot)\|_{\infty}+2K+\frac{\delta}{2},
$$
which satisfy \eqref{kappax} and \eqref{kappay}, and we have
$$
K = \frac{\kappa_{x}+\kappa_{y}}{2},
$$
and
\begin{equation}
\begin{aligned}
    \kappa_{x}-\kappa_{y} & = -\sup_{t\in[0,\infty)}\lambda_{\max} (b_{1}(t)+b_{1}(t)^{\top})-2\|b_{2}(\cdot)\|_{\infty}-2K-\delta\\
    & = -\sup_{t\in[0,\infty)}\lambda_{\max} (b_{1}(t)+b_{1}(t)^{\top})-2\|b_{2}(\cdot)\|_{\infty}-\delta\\
    &\quad -2\left(-\frac{1}{2}\sup_{t\in[0,\infty)}\lambda_{\max} (b_{1}(t)+b_{1}(t)^{\top}) -\|b_{2}(\cdot)\|_{\infty}-\frac{(\|\sigma_{1}(\cdot)\|_{\infty}+\|\sigma_{2}(\cdot)\|_{\infty})^{2}}{2}-\delta\right)\\
    & = (\|\sigma_{1}(\cdot)\|_{\infty}+\|\sigma_{2}(\cdot)\|_{\infty})^{2}+\delta\\
    & > (\|\sigma_{1}(\cdot)\|_{\infty}+\|\sigma_{2}(\cdot)\|_{\infty})^{2}+\delta_{0},
\end{aligned}
\end{equation}
for any $0<\delta_{0}<\delta$.
Therefore, letting $\varepsilon_{\sigma} = \varepsilon_{z} = \delta_{0}$ in \eqref{sigma} and \eqref{ykppay}, we have shown that if the condition \eqref{b1b2condition} is satisfied, we can choose $\kappa_{x},\kappa_{y}$ satisfying \eqref{kappax} and
\eqref{kappay} respectively, such that
\begin{equation}\label{5}
\kappa_{x}-\kappa_{y}> \operatorname{max}\{l_{\sigma},l_{z}\} \quad  \text{and} \quad K = \frac{\kappa_{x}+\kappa_{y}}{2}.
\end{equation}
 The proof is finished.
\end{proof}
At the end of this section, let us recall the infinite horizon control problems studied in \cite{bayraktar2023solvability} and \cite{yu2021} and provide some remarks to make a comparison.
\begin{remark}\label{K}
Bayraktar and Zhang \cite{bayraktar2023solvability} considered an infinite horizon mean field control problem as follows.\\
Minimize the problem
\begin{equation}\label{1}
J(\alpha):=\mathbb{E}\left[\int_0^{\infty} e^{2Kt} f\left(t, X_t, \mathcal{L}\left(X_t\right), \alpha_t\right) d t\right],
\end{equation}
subject to 
\begin{equation}\label{2}
\left\{\begin{aligned}
    d X_t &=b\left(t, X_t, \mathcal{L}\left(X_t\right), \alpha_t\right) d t+\sigma d W_t, \\
X_0 &= \xi.
\end{aligned}\right.
\end{equation}
When solving the Hamiltonian system, they let $b(t, x, \mu, \alpha)=b_0(t)+b_1(t)x+\bar{b}_1(t)\bar{\mu}+b_2(t)\alpha$ and proposed the following condition
\begin{equation}\label{b}
\left|\bar{b}_1(t)\right| \leq l \quad \text{ and } -\max _t b_1(t) \geq l+K,  
\end{equation}
which can be easily verified to correspond to our condition \eqref{b1b2condition} on parameter $K$. Not only did they need \eqref{b}, they also needed to supplement another  constraint for $K$ to solve the mean field control problem \eqref{1}-\eqref{2} (see \cite[Theorem 3.1]{bayraktar2023solvability}). Thus, compared with \cite{bayraktar2023solvability}, we weaken the conditions on the values of parameter $K$ and extend the results to a larger space.
\end{remark}
\begin{remark}\label{S}
    Wei and Yu \cite{yu2021} studied an infinite horizon LQ optimal control problems as follows.
    Minimize:
    \begin{equation}\label{LQ1}
J(\alpha)=\frac{1}{2} \mathbb{E} \int_0^{\infty} e^{2 K s}[\langle Q(s) x(s), x(s)\rangle+2\langle S(s) x(s), \alpha(s)\rangle+\langle R(s) \alpha(s), \alpha(s)\rangle] d s,
\end{equation}
subject to
\begin{equation}{\label{LQ2}}
\left\{\begin{aligned}
 d x(s) &=[A(s) x(s)+B(s) \alpha(s)]ds+\left[C(s) x(s)+D(s) \alpha(s)\right]d W_{s},\quad  s \in[0, \infty), \\
x(0) & =x.
\end{aligned}\right.
\end{equation}
They pointed out that if $S(\cdot)$ is a nonzero matrix, the values of parameter $K$ will be adjusted accordingly. It is obvious that the optimization problem \eqref{minimization} covers LQ problem \eqref{LQ1}-\eqref{LQ2}. Thus it follows from Theorem \ref{theorem mfc solvability} that if there exists a constant $\lambda>0$ such that $$
Q(\cdot)-S(\cdot)^{\top} R(\cdot)^{-1} S(\cdot)- \lambda I \geq 0,
$$ and  
$$
K <-\frac{1}{2}\sup_{t\in[0,\infty)}\lambda_{\max} (A(t)+A(t)^{\top}) -\frac{\|C(\cdot)\|_{\infty}^{2}}{2},
$$
the Hamiltonian system corresponding to the LQ problem \eqref{LQ1}-\eqref{LQ2} admits a unique solution $(X,Y,Z)\in L^{2,K}_{\mathbb{F}}(0,\infty;\mathbb{R}^{n+n+n\times d})$. Therefore, we can obtain the same solvability results for the LQ problem \eqref{LQ1}-\eqref{LQ2} whether the cross term coefficient $S(\cdot)$ is equal to zero or not. The root causing
this different phenomenon is that the values of $S(\cdot )$ will directly affect the monotonicity condition proposed in \cite{yu2021}, but it will not affect our condition since it only appears in optimal control and  we do not need to consider its monotonicity under our condition. 
\end{remark}
\bibliographystyle{siam}
\bibliography{bibliography}
\end{document}